\documentclass[final,1p,times]{elsarticle}

\usepackage{amssymb,amsthm,amsmath}

\newtheorem{thm}{Theorem}     
\newtheorem{prop}[thm]{Proposition}    

\journal{ArXiv.org}

\begin{document}

\begin{frontmatter}


\title{An upper bound on the Kolmogorov widths of a certain family of integral operators\tnoteref{label1}}
\tnotetext[label1]{This paper is dedicated to the memory of Eduard S.~Belinsky (1947--2004).}

\author{Duaine Lewis}
\ead{duaine.lewis@mycavehill.uwi.edu}

\author{Bernd Sing\corref{cor1}}
\ead{bernd.sing@cavehill.uwi.edu}
\ead[url]{http://www.bb-math.com/bernd}

\cortext[cor1]{Corresponding author}

\address{Department of Mathenmatics, The University of the West Indies, Cave~Hill, P.O.~Box~64, Bridgetown, St~Michael, BB11000, Barbados}

\begin{abstract}
We consider the family of integral operators $(K_{\alpha}f)(x)$ from $L^p[0,1]$ to $L^q[0,1]$ given by
$$(K_{\alpha}f)(x)=\int_0^1(1-xy)^{\alpha -1}\,f(y)\,\operatorname{d}\!y, \qquad 0<\alpha<1.$$
The main objective is to find upper bounds for the Kolmogorov widths, where the $n$th Kolmogorov width is the infimum of the deviation of $(K_{\alpha}f)$ from an $n$-dimensional subspaces of $L^p[0,1]$ (with the infimum taken over all $n$-dimensional subspaces), and is therefore a measure of how well $K_{\alpha}$ can be approximated. We find upper bounds for the Kolmogorov widths in question that decrease faster than $\exp(-\kappa \sqrt{n})$ for some positive constant $\kappa$.
\end{abstract}


\end{frontmatter}


\section{Introduction}

Integral operators acting on Hilbert spaces are classical objects studied in functional analysis as well as approximation theory. E.g., extending results and methods in  \cite{BS2,W}, Laptev \cite{Lap} considered the compactness properties and singular values of the integral operator $T: L^2[0,1]\to L^2[0,1]$ given by
$$ (Tf)(x) = \int_0^1 \frac{x^{\beta} y^{\gamma}}{(x+y)^{\alpha}}\,f(y)\,\operatorname{d}\!y,$$
where $\alpha>0$, $\beta,\gamma>-\frac12$ and $\beta+\gamma-\alpha+1>0$; it is shown that the singular values $\lambda_n$ behave like $\exp(-c\sqrt{n})$ for some positive contant $c$. Belinsky and Linde~\cite{BL} investigated the compactness properties of the integral operators $S_\alpha: L^p[0,1]\to L^q[0,1]$, i.e., an integral operator between more general spaces, given by
$$(S_{\alpha}f)(x)=\int_0^1(x+y)^{\alpha -1}\,f(y)\,\operatorname{d}\!y,$$
where $1\le p$,  $q\le \infty$ and $\alpha > (1/p - 1/q)_+=\max\{0,1/p - 1/q\}$. It is shown in \cite{K,BL} that the Kolmogorov widths of these operator $S_{\alpha}$ tend to zero faster than $\exp(-c_{\alpha}\sqrt{n})$ for some constant $c_{\alpha}=c_{\alpha}(p,q)>0$. Since $L^p$-spaces are in general not Hilbert spaces (for which singular values are defined only), Kolmogorov widths serve as a substitute for singular values.

In this paper we consider the integral operator $(K_{\alpha}f)(x): L^p[0,1]\to L^q[0,1]$ where 
\begin{equation}\label{eq:iop}
(K_{\alpha}f)(x):=\int_0^1(1-xy)^{\alpha -1}\,f(y)\,\operatorname{d}\!y, \qquad 0<\alpha<1.
\end{equation}
This integral operator belongs to several well-studied operator algebras: it is a Hilbert-Schmidt operator, a Schatten class operator for every $0<p<\infty$, and a bounded $L^p[0,1]$ operator for $1<p<\infty$, see Section~\ref{sec:properties}. The main purpose of the paper is to approximate this operator and obtaining upper bounds for the Kolmogorov widths, which we will explain next.

From the abstract point of view, approximation by polynomials or by trigonometric polynomials is a very special process. 
It is natural to try approximation by other systems of functions, compare \cite{DL,Kor,LGM}. For a given class of functions $A$, we can even try to find a `most favorable' system of approximation $A'$. We note that if $A$ consists of a single function, the degree (or error) of approximation  of the function $f$  is zero if $f$ itself is included in the system $A'$.  In general, let $X$ be a Banach space, and $A$ and $A'$ two subsets of $X$. The \emph{deviation} from $A$ to $A'$ is the number
$$E(A, A')_X=\sup_{f\in A}\{\inf_{g\in A'}\|f-g\|\}.$$
The deviation shows how well the ``worst" elements of $A$ can be approximated by $A'$. 

Now, let $A'=X_n$ be an $n$-dimensional subspace of $X$ spanned by the elements $\phi_1,\cdots,\phi_n$. The number $E(A,X_n)_{X}$
is the degree of approximation of the class $A$ by the set of linear combinations $a_1\phi_1+\cdots+a_n\phi_n$.  It was Kolmogorov's
idea \cite{L} to consider the infimum of the deviation for all $n$-dimensional subspaces $X_n$ of $X$:
The number
$$d_n(A,X)=\inf_{\substack{X_n\subset X,\\ \dim X_n=n}}E(A,X_n)_X, \hspace{10mm}   n=0,1,2,\cdots. $$
is called the \emph{Kolmogorov $n$-width} of $A$ in $X$. For a (compact) linear operator $K: X\to Y$ between Banach spaces $X$ and $Y$, we set
\begin{equation*}
d_n(K) = d_n(K: X\to Y) \ =\ d_n(K(B_X), Y),
\end{equation*}
where $B_X$ denotes the unit ball in $X$; i.e., the Kolmogorov $n$-width of an operator $K$ is the Kolmogorov $n$-width of the image of the unit ball under $K$.

Kolmogorov widths $d_n(A,X)$ measure the extent to which $A$ may be approximated by $n$-dimensional subspaces of $X$ and can thus help identifying optimal subspaces, see \cite{Pin86,Pin03} for an elementary introduction and the books \cite{LGM,Pietsch,Pietsch2,Pinkus,Tik90} for thorough expositions of Kolmogorov $n$-widths and overview of the various situations where they appear. In fact, Kolmogorov $n$-widths are an instance of the general class of $s$-numbers to which also \emph{Gel'fand numbers}, \emph{approximation numbers}, \emph{Hilbert numbers}, etc.\ belong, see \cite[Chapter 13]{LGM} and \cite[Chapter 11]{Pietsch} for details and the relationships between them.

As already mentioned above, Kolmogorov $n$-widths of an operator $K$ may serve as a substitute for singular values since they coincide with them in a Hilbert space.  To simplify matters, let $K: L^2[0,1]\to L^2[0,1]$ be a real-valued Hilbert-Schmidt operator with kernel $k(x,y)$ defined on $[0,1]\times[0,1]$, i.e., $(Kf)(x) = \int_0^1 k(x,y)\,f(y)\,\operatorname{d}\!y$, also compare Section~\ref{sec:properties}. Then the \emph{singular values} of the operator $K$ are defined as the square roots of the eigenvalues of the self-adjoint, non-negative, compact operator $(K'K)$ induced by the kernel $(k'k)(x,y)=\int_0^1 k(z,x)\,k(z,y)\,\operatorname{d}\!z$; i.e., if $(K'K)\varphi_i = \lambda_i\varphi_i$, $i=1,2,\ldots,$ where $\lambda_1\ge \lambda_2\ge \ldots$ enumerate the nonzero (positive) eigenvalues of $(K'K)$ with algebraic multiplicity, then the $i$-th singular value of $K$ is defined as $\sqrt{\lambda_i}$. In this situation, we have $d_n(K: L^2[0,1]\to L^2[0,1]) = \sqrt{\lambda_{n+1}}$, and an optimal $n$-dimensional subspace for $K(B_{L^2[0,1]})$ in $L^2[0,1]$ is spanned by $K\varphi_1,\ldots, K\varphi_n$, see \cite[Theorem I.2]{Pinkus}. Relationships between singular values and $s$-numbers in more general settings can be found in \cite{Pietsch,Pietsch2}, the survey article \cite{BS} specifically deals with estimates of $s$-numbers for integral operators.

The main result of this article are the following upper bounds on the Kolmogorov widths of the operators $K_{\alpha}$.

\begin{thm}\label{thm:main}
Let $K_{\alpha}$ be the integral operator 
$$(K_{\alpha}f)(x)= \int_0^1 (1-xy)^{\alpha -1} \,f(y)\,\operatorname{d}\!y\qquad\text{where}\quad 0<\alpha <1.$$ 
Then the Kolmogorov $n$-widths of $K_{\alpha}$ are asymptotically bounded as follows:
\begin{eqnarray*}
d_{n}\left(K_{\alpha}: L^p[0,1]\to L^{\infty}[0,1]  \right) & \le\ O\left( 2^{-\kappa_1\, \sqrt{n}} \right) & \text{if }\frac1{p}<\alpha<1, \\
d_{n}\left(K_{\alpha}: L^p[0,1]\to L^{r}[0,1]  \right) & \le\ O\left( 2^{-\kappa_2\, \sqrt{n}} \right) & \text{if }\alpha=\frac1{p},\ 1\le r<\infty, \\
d_{n}\left(K_{\alpha}: L^p[0,1]\to L^{r}[0,1]  \right) & \le\ O\left( 2^{-\kappa_3\,\sqrt{n}} \right) & \text{if }1-\frac1{r}<\alpha<\frac1{p},
\end{eqnarray*}
for some positive constants $\kappa_1<\frac1{\sqrt{2}}\left(\alpha-\frac1{p}\right)$, $\kappa_2<\frac1{\sqrt{2}}\,\frac1{r}$ and $\kappa_3<\frac1{\sqrt{2}}\,\left(1-\frac1{p}\right)$.
\end{thm}

This article is orgainzed as follows: In the next section, we consider some properties of the integral operators $K_{\alpha}$, establishing that they belong to various well-studied operator algebras. Section~\ref{sec:proof} contains the proof of Theorem \ref{thm:main}; it closely follows a corresponding proof for the integral operators $S_{\alpha}$ in Belinsky and Linde \cite{BL} mentioned above. We then add a remark on entropy numbers in Section~\ref{sec:entropy} which is an alternative way to measure the massivity of the set $K_{\alpha}(B_{L^p[0,1]})$. In Section~\ref{sec:example} we consider an example to make the approximation  obtained during the proof more concrete -- also compare Fig.~\ref{fig:case1n} and the Mathematica code in~\ref{sec:mathematica} -- before concluding the article with an outlook.

\section{Properties of the integral operators $K_\alpha$}

\subsection{Some properties}\label{sec:properties}

Recall that a \emph{Hilbert-Schmidt operator} $K$ is a linear operator on a Hilbert Space for which the Hilbert-Schmidt norm $\|K\|_2$ is finite, compare \cite[Theorem 3.1.5 \& Remark 3.1.6]{Z}, and also see \cite[Theorem 3.8.5]{Simon}. Considering the Hilbert space $L^2[0,1]$ here, we obtain the following result.

\begin{prop}\label{prop:HS1}
The integral operator in Eq.~\eqref{eq:iop} on $L^2[0,1]$ is a Hilbert-Schmidt operator with norm $\|K_\alpha\|_2^{}=\frac{\pi}{\sqrt{6}}$ if $\alpha=\frac12$ and $\|K_\alpha\|_2^{}=\sqrt{\frac{\gamma+\Psi(2\alpha)}{2\alpha-1}}$ otherwise, where $\gamma$ is the Euler-Mascheroni constant and $\Psi$ is the digamma function.
\end{prop}

\begin{proof}
We first note that the double integral in question is improper, i.e., 
\begin{eqnarray*}
\int_0^1 \int_0^1\left((1-xy)^{\alpha -1}\right)^2 \operatorname{d}\!x\,\operatorname{d}\!y &  = & \lim_{(s,t)\to (1,1)^-} \int_0^s\int_0^t (1-xy)^{2\,\alpha -2} \operatorname{d}\!x\,\operatorname{d}\!y. 
\end{eqnarray*}
We calculate that
$$ \int_0^t (1-xy)^{2\,\alpha -2} \operatorname{d}\!x = \begin{cases}
t & \mbox{if } y=0, \\
\frac{1-(1-t\,y)^{2\,\alpha-1}}{(2\,\alpha-1)\,y} & \mbox{if } 0<y<1 \mbox{ and } \alpha\neq \frac12 ,\\
-\frac{\log(1-t\,y)}{y} &  \mbox{if } 0<y<1 \mbox{ and } \alpha = \frac12,
\end{cases}$$
and note that this function is continuous at $y=0$.

In the $\alpha=\frac12$ case, we have 
$$  \int_0^s -\frac{\log(1-t\,y)}{y} \operatorname{d}\!y =  \int_0^{s\,t} -\frac{\log(1-u)}{u} \operatorname{d}\!u = \operatorname{Li}_2(s\,t),$$
using the dilogarithm $\operatorname{Li}_2(z)$ defined either by $ \operatorname{Li}_2(z) =\int\limits_0^z -\frac{\log(1-u)}{u}\,\operatorname{d}\!u$ or by the series expansion $\operatorname{Li}_2(z) = \sum\limits_{k=1}^{\infty} \frac{z^k}{k^2}$ for $|z|\le 1$  (where we note that $\operatorname{Li}_2(1)=\zeta(2)=\frac{\pi^2}6$), see \cite[\S 25.12(i)]{NIST}. Consequently, in this case we have 
$$\int_0^1 \int_0^1\left((1-xy)^{\alpha -1}\right)^2 \operatorname{d}\!x\,\operatorname{d}\!y = \lim\limits_{(s,t)\to(1,1)^-} \operatorname{Li}_2(s\,t)=\operatorname{Li}_2(1)
= \frac{\pi^2}6.$$

In the $\alpha\neq\frac12$ case, we have
$$ \int_0^s \frac{1-(1-t\,y)^{2\,\alpha-1}}{(2\,\alpha-1)\,y} \operatorname{d}\!y =   s t\cdot\int_0^1 \frac{1-(1-s t\,u)^{2\,\alpha-1}}{(2\,\alpha-1)\,st\,u} \operatorname{d}\!u =s\,t\cdot \mbox{}_3F_2(1,1,2-2\alpha;2,2;s\,t),$$
using the negative binomial series (with non-integral exponent) and the generalized hypergeometric function ${}_3F_2(1,1,2-2\alpha;2,2;z)$ defined by the series expansion
$$ {}_3F_2(1,1,2-2\alpha;2,2;z) = 1+\sum_{k=1}^{\infty}\frac{(2-2\alpha)(3-2\alpha)\cdots(k+1-2\alpha)}{k+1} \frac{z^k}{(k+1)!}$$ which is converging for $|z|\le 1$ and diverging for $|z|>1$ (it is absolutely convergent on $|z|=1$), see \cite[Section 44]{Rai60}.
In the limit $(s,t)\to (1,1)^-$, we obtain
\begin{multline*}  
\lim_{(s,t)\to (1,1)^-} s t \cdot\int_0^{1} \frac{1-(1-s t \,u)^{2\,\alpha-1}}{(2\,\alpha-1)\,s t \,u} \operatorname{d}\!u 
= \int_0^{1} \frac{1-v^{2\,\alpha-1}}{(2\,\alpha-1)\,(1-v)} dv =\\ =  1+\sum_{k=1}^{\infty}\frac{(2-2\alpha)(3-2\alpha)\cdots(k+1-2\alpha)}{k+1} \frac1{(k+1)!}  = \frac{\gamma+\Psi(2\alpha)}{2\alpha-1}
\end{multline*}
where $\gamma\approx 0.577\,215\,664\,\ldots $ is the Euler-Mascheroni constant, and $\Psi$ is the digamma function (also known as psi function) defined by $\Psi(z)=\frac{d}{dz}\log\Gamma(z)$ (the logarithmic derivative of the gamma function), see \cite[Formula 5.9.16]{NIST}.

Overall, we therefore have (note that this is continuous at $\alpha=\frac12$)
\begin{eqnarray*}
\|K_{\alpha}\|_2^2\ =\ \int_0^1 \int_0^1\left((1-xy)^{\alpha -1}\right)^2 \operatorname{d}\!x\,\operatorname{d}\!y  & = & \begin{cases} \frac{\pi^2}6 & \mbox{if } \alpha=\frac12,\\  \frac{\gamma+\Psi(2\alpha)}{2\alpha-1} & \mbox{if } 0<\alpha<1,\  \alpha\neq \frac12,\end{cases} 
\end{eqnarray*}
which is finite for $\alpha\in(0,1)$ not least since the digamma function is holomorphic on\linebreak $\mathbb{C}\setminus\{0,-1,-2,-3,-4,\ldots\}$, compare \cite[\S 5.2(i)]{NIST}.
\end{proof}

While the previous proposition yields the Hilbert-Schmidt norm of the integral operator in question, it is actually easy to obtain an upper bound on the norm in any $L^p$ space using \emph{Schur's theorem}, see \cite[Section \textsection 3.2]{Z}, and thus showing that it is a bounded integral operator.

\begin{prop}\label{prop:HS2} 
The integral operator in Eq.~\eqref{eq:iop} is bounded on $L^p[0,1]$ for $1<p<+\infty$ with norm less than or equal to $1/\alpha$.
\end{prop}

\begin{proof}
We have
$$\int_0^1 (1-xy)^{\alpha -1} \operatorname{d}\!y =\begin{cases}  \frac{1-(1-x)^{\alpha}}{\alpha\,x} & \text{if } x>0, \\ 1& \text{if } x=0,\end{cases}$$
this function is continuous at $x=0$ and thus on $[0,1]$. Furthermore, one can show that this function is increasing on $(0,1)$ (for $0<\alpha<1$). Thus, it attains its maximum at $x=1$ and we have
$$0\le \int_0^1 (1-xy)^{\alpha -1} \operatorname{d}\!y \le \frac1{\alpha}.$$
The claim now follows directly from \cite[Theorem 3.2.2]{Z} (also compare \cite[Corollary 3.2.3]{Z} for the $L^2$ case).
\end{proof}

Having established that $K_{\alpha}$ is a Hilbert-Schmidt operator, we look at the following generalisation which connects integral operators with sequence spaces, compare \cite[Section \textsection 1.4]{Z}: Consider a bounded linear operator $K$ on the separable Hilbert space $L^2[0.1]$. Let $\lambda_n$ be the \emph{$n$-th singular value} of $K$ where $\lambda_1\ge \lambda_2\ge \lambda_3\ge \ldots$ (we note that $K_{\alpha}$ is a positive, self-adjoint operator, so that the singular values are the eigenvalues of $K_{\alpha}$). We say that an operator $K$ belongs to the \emph{Schatten $p$-class} if $\left(\sum_{n\ge 1} \lambda_n^p\right) < \infty$ (for $0<p<\infty$); in particular, for $1\le p<\infty$, we can define the \emph{Schatten $p$-norm} of $K$ by
\begin{equation*}
\|K\|_{S^p} = \left(\sum_{n\ge 1} \lambda_n^p\right)^{1/p},
\end{equation*}
i.e.,\ by the $\ell^p$-norm on the sequence of singular values (for $0<p<1$ this only yields a quasinorm); in this case, $K$ belongs to the Schatten $p$-class iff its Schatten $p$-norm is finite. Note that $\|K\|_2 = \|K\|_{S^2}$ and thus a Schatten $2$-class operator is also said to belong to the \emph{Hilbert-Schmidt class}. Also note that an operator $K$ belonging to the Schatten $1$-class is a \emph{trace class operator}. Further details can be found in \cite[Sections 3.7--3.8]{Simon}.

We now remark that $K$ actually belongs to the Schatten class $S_p$ for every $0<p<\infty$. This is an immediate from \cite[Theorem 4]{P}, a variant of the so-called Luecking Theorem in \cite{Lue}. 

\begin{prop}\label{prop:KSchatten}
Let $K_\alpha$ be the integral operator in $L^2[0,1]$ given in Eq.~\eqref{eq:iop}.
Then the operator $K_{\alpha}$ belongs to the Schatten class $S_p$ for every $0<p<\infty$.\qed
\end{prop}

\subsection{Re-writing the integral operators}

We now look at the integral operator 
$$(K_{\alpha}f)(x)=\int_0^1 (1-xy)^{-\alpha -1}\,f(y)\,\operatorname{d}\!y.$$
For this integral operator $(K_{\alpha}f)(x)$ we let $-\beta=\alpha -1$, thus $0<\beta<1$, and the operator becomes
$$(K_{1-\beta}f)(x)=\int_0^1 (1-xy)^{-\beta}\,f(y)\,\operatorname{d}\!y.$$
We also change variables: Let $u=1-x$ and $v=1-y$ (therefore $-\operatorname{d}\!y=\operatorname{d}\!v$); then the operator becomes
$$(K_{1-\beta}f)(1-u)=\int_0^1\frac{f(1-v)}{(u+v-uv)^{\beta}}\,\operatorname{d}\!v.$$
Using the notation $\widetilde{f}(z)=f(1-z)$, this can be written as
$$\widetilde{(K_{1-\beta}f)}(u)=\int_0^1\frac{\widetilde{f}(v)}{(u+v-uv)^{\beta}}\,\operatorname{d}\!v.$$
%
%
%

Note the the kernel $\frac1{(u+v-uv)^{\beta}}$ of this integral operator has a singularity at $(0,0)$; it is defined for all $(u,v)\in [0,1]\times[0,1]\setminus\{(0,0)\}$, and we have $\lim\limits_{(u,v)\to(0,0)^+} \frac{1}{(u+v-uv)^{\beta}} = +\infty$.

\section{Proof of Theorem 1}\label{sec:proof}

\subsection{Outline of the Proof}

We approximate $\widetilde{(K_{1-\beta}f)}$ by a rational function of order $n$ (i.e., degree $n-1$) on a partition of $[0,1]$ into finitely many (namely, $n+1$ many) intervals. More precisely, we will approximate $u\mapsto\widetilde{(K_{1-\beta}f)}(u)$ for $u$ on each of the intervals $\left( 2^{-k-1},2^{-k} \right]$ for $k=0,\ldots,n-1$ and the interval $\left[0,2^{-n}\right]$, i.e.,\ the endpoints of the intervals used here are the dyadic fractions $2^{-k}$ for $k=0,\ldots, n-1$, and $0$, see Figure~\ref{fig:dyadic} for the case $n=4$. 

\medskip

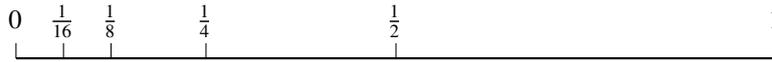
\begin{figure}[htp]
{\setlength{\unitlength}{1cm}
\centerline{\begin{picture}(11,.5)
{\thicklines%
{%
\put(.5,0){\line(1,0){10}}
}}%
\put(.5,0){\line(0,1){.2}}
\put(1.125,0){\line(0,1){.2}}
\put(1.75,0){\line(0,1){.2}}
\put(3,0){\line(0,1){.2}}
\put(5.5,0){\line(0,1){.2}}
\put(10.5,0){\line(0,1){.2}}
\put(0.4,.4){$0$}
\put(0.94,.4){$\frac1{16}$}
\put(1.635,.4){$\frac18$}
\put(2.875,.4){$\frac14$}
\put(5.375,.4){$\frac12$}
\put(10.4,.4){$1$}
\end{picture}}}
\caption{Intervals in the case $n=4$ used the calculations.\label{fig:dyadic}}
\end{figure}

To this end, we first consider the case $u\in \left[0,2^{-n}\right]$ in Section~\ref{subsec:case1} where we split $\widetilde{(K_{1-\beta}f)}(u)$ into two parts, namely the integral from $0$ to $2^{-(n-1)}$ (the leftmost interval) and the integral from $2^{-(n-1)}$ to $1$. Next, in Section~\ref{subsec:case2}, we consider the case $u\in\left( 2^{-k-1},2^{-k} \right]$ for some $k=0,\ldots,n-1$ where the integral  $\widetilde{(K_{1-\beta}f)}(u)$ is (typically) split into three parts: the integral from $0$ to $2^{-(n-1)}$ (the leftmost interval),  the integral over the interval $\left( 2^{-k-1},2^{-k} \right]$ the variable $u$ falls into, and the integral over the remaining $n-1$ intervals of the form $\left( 2^{-k-1},2^{-k} \right]$. 

Depending on the three cases $0<\beta q <1$, $\beta q=1$ and $\beta q>1$, we then consider the error of approximation made in Sections \ref{subsec:error1}, \ref{subsec:error2}, respectively \ref{subsec:error3}. Since our approximation is of dimension $2n^2+n$, see Section~\ref{subsec:dim}, this then establishes the main theorem.

We first start with some result we will frequently use in what follows.

\subsection{A Taylor series and an asymptotic result}

Besides H\"{o}lder's inequality and the generalized version of Bernoulli's inequality, we will often make use of the Taylor series for $(1+x)^{-\beta}$, 
$$
(1+x)^{-\beta} = \sum_{j=0}^{\infty} (-1)^j \frac{\Gamma(\beta+j)}{\Gamma(\beta)\cdot\Gamma(j+1)}\, x^j,
$$
which converges for $|x|<1$. Thus, the partial sum with remainder term of  degree $n$ (we choose the Lagrangian form of the remainder term here) is given by
\begin{equation}\label{eq:taylor1}
 (1+x)^{-\beta} = \left(\sum_{j=0}^{n-1} (-1)^j \frac{\Gamma(\beta+j)}{\Gamma(\beta)\cdot\Gamma(j+1)}\, x^j\right) + (-1)^n\,\frac{\Gamma(\beta+n)}{\Gamma(\beta)\,\Gamma(n+1)}\,(1+\theta)^{-\beta-n}\,x^n,
\end{equation}
where the real number $\theta$ is between $0$ and $x$ (i.e., either $\theta\in(0,x)$ if $x>0$ or $\theta\in(x,0)$ if $x<0$), and the Taylor polynomial of degree $(n-1)$ is given by
$$ P_{n-1}(x) = \sum_{j=0}^{n-1} (-1)^j \frac{\Gamma(\beta+j)}{\Gamma(\beta)\cdot\Gamma(j+1)}\, x^j.$$

In the following, we will also need the asymptotic behaviour of $\frac{\Gamma(\beta+n)}{\Gamma(n+1)}$. From the asymptotic formula 
$$ \frac{\Gamma(z+a)}{\Gamma(z+b)} = z^{a-b}\left(1+\frac{(a-b)(a+b-1)}{2z}+O(|z|^{-2})\right)$$
for arbitrary constants $a,b$ and $|\arg(z)|<\pi-\delta$ for some $0<\delta\ll 1$ (see \cite[p.~15]{Lebedev}), we get
\begin{equation}\label{eq:Gammasym}
 \frac{\Gamma(n+\beta)}{\Gamma(n+1)} = n^{\beta-1} \left[1-\frac{\beta(1-\beta)}{2n}+O(n^{-2}) \right]. 
\end{equation}

\subsection{Approximation if $u\in\left[0,2^{-n}\right]$.\label{subsec:case1}}

We first  consider the case where $u$ in $\widetilde{(K_{1-\beta}f)}(u)$ falls into the leftmost interval; in this case, we write $\widetilde{(K_{1-\beta}f)}(u)$ as a sum of two terms:
$$\widetilde{(K_{1-\beta}f)}(u)=\int_0^{2^{-(n-1)}} \frac{f(v)}{(u+v-uv)^{\beta}}\,\operatorname{d}\!v +\int_{2^{-(n-1)}}^1 \frac{f(v)}{(u+v-uv)^{\beta}}\,\operatorname{d}\!v.$$

For the first of these two integrals, i.e., $\int_0^{2^{-(n-1)}} (u+v-uv)^{-\beta}\,f(v)\,\operatorname{d}\!v$, we first apply H\"{o}lder's inequality together with the integral estimate $\left(\int_0^x |f(v)|^p\,\operatorname{d}\!v\right)^{1/p} \le \|f\|_p$ for any $0\le x\le 1$, to obtain 
\begin{multline*}
\left| \int\limits_0^{2^{-(n-1)}} \frac{f(v)}{(u+v-uv)^{\beta}}\,\operatorname{d}\!v \right| 
 \  \le\ \|f\|_p\cdot \left(\int\limits_0^{2^{-(n-1)}} \frac{\operatorname{d}\!v}{(u+v-uv)^{\beta q}}\right)^{1/q}
\\ = \left\{ \begin{array}{ll}
\|f\|_p \left(\frac{\ln\left(1+\frac{1-u}{2^{n-1} u}\right)}{1-u}\right)^{1/q} & \text{if } q\beta=1,\\
\|f\|_p \left(\frac{\left[u+\frac{1-u}{2^{n-1}}\right]^{1-q\beta} -u^{1-q\beta}}{(1-q\beta)(1-u)}\right)^{1/q} & \text{if } q\beta\neq1,\\
\end{array}\right\}
\stackrel{(\star)}{\le}
  \begin{cases} 
O\left(\| f\|_p\cdot 2^{-n\left(\frac1{q}-\beta\right)}\right) &  \text{if } 0<q\,\beta<1,\\
O\left(\| f\|_p\cdot \left(\frac{\ln 3-\ln\left({2^{n}\,u}\right)}{1-2^{-n}} \right)^{1/q}\right) &  \text{if } q\,\beta = 1, \\
O\left(\| f\|_p\cdot 2^{-n/q}\,u^{-\beta}\right) & \text{if } q\,\beta > 1.\\
\end{cases}\\
\end{multline*}
In step $(\star)$ we used the generalized version of Bernoulli's inequality (namely that for $h>-1$ we have $(1+h)^r\le 1+r\,h$ if $0\le r\le 1$ and $(1+h)^r\ge 1+r\,h$ if $r\le 0$ or $r\ge 1$) together with some straightforward estimates.

 The second integral, i.e., $\int_{\frac1{2^{n-1}}}^1\frac{f(v)}{(u+v-uv)^\beta}\,\operatorname{d}\!v$, is represented in the form
$$\int_{\frac1{2^{n-1}}}^1\frac{f(v)}{v^\beta (1+{\frac{u}{v}}-u)^\beta}\,\operatorname{d}\!v=\int_{\frac1{2^{n-1}}}^1\frac{f(v)}{v^\beta \left[1+u\,\left(\frac1{v}-1\right)\right]^\beta}\,\operatorname{d}\!v.$$
Approximating $\left[1+u\,\left(\frac1{v}-1\right)\right]^{\beta}$ by partial sums of the Taylor series above in Eq.~\eqref{eq:taylor1} (where we set $x=u\,\left(\frac1{v}-1\right)=\left(\frac{u}{v}-u\right)$), we obtain
\begin{multline*}
\left|\,\int\limits_{\frac1{2^{n-1}}}^1 f(v)\,v^{-\beta}\,\left[(1+u(1/v-1)]^{-\beta}-P_{n-1}(u(1/v-1)\right]\,\operatorname{d}\!v\right | \\
= \frac{\Gamma(\beta +n)}{\Gamma(\beta)\Gamma(n+1)}\, \left|\,\int\limits_{\frac1{2^{n-1}}}^1 f(v)v^{-\beta}\,\frac{u^n \left(\frac1{v}-1\right)^n}{(1+\theta)^{\beta+n}}\,\operatorname{d}\!v\right|
\le\frac{\Gamma(\beta +n)}{\Gamma(\beta)\Gamma(n+1)}\cdot{\frac1{2^n}}\, \left|\,\int\limits_{\frac1{2^{n-1}}}^1 f(v)v^{-\beta}\,\operatorname{d}\!v\right|,
\end{multline*}
where in the last step we used that $\theta>0$ and $\left|u\left(\frac1{v}-1\right)\right| \le \frac1{2^n} (2^{n-1}-1) < \frac12$.

For the integral in this last expression we obtain by H\"{o}lder's inequality (together with some other straightforward estimates) that 
\begin{multline*}
\left |\int_{\frac1{2^{n-1}}}^1 f(v)v^{-\beta}\,\operatorname{d}\!v\right| \le   \left\{ \begin{array}{ll} 
\|f\|_p \left( (n-1) \ln 2\right)^{1/q} & \text{if } \beta q=1,\\
\|f\|_p \,\left[ \frac{1-2^{-(n-1)(1-\beta q)}}{1-\beta q}\right]^{1/q} & \text{if } \beta q\neq 1,
\end{array} \right\} \\
\le \begin{cases}
\left[\frac{1}{1-\beta q}\right ]^{\frac1{q}}\,\|f\|_p & \text{if } 0<\beta q<1,\\
\left(\ln 2\right)^{\beta}\,\|f\|_p\, n^{\beta} & \text{if } \beta q= 1,\\
\frac1{(\beta q-1)^{\frac1{q}}}\,\|f\|_p\,{2^{n\left(1-\frac1{q}\right)}} & \text{if } \beta q >1,\\
\end{cases}
\end{multline*}
where for the last inequality we note that if $\beta q>1$ then $0<\beta-\frac1{q}<1-\frac1{q}\le 1$. Thus, we have
$$ \left |\int_{\frac1{2^{n-1}}}^1 f(v)v^{-\beta}\,\operatorname{d}\!v\right| \ \le \ 
 \begin{cases}
\mbox{const}\cdot\|f\|_p & \text{if } 0<\beta q<1,\\
\mbox{const}\cdot\|f\|_p\, n^{\beta} & \text{if } \beta q= 1,\\
\mbox{const}\cdot\|f\|_p\,{2^{n\left(1-\frac1{q}\right)}} & \text{if } \beta q >1.\\
\end{cases} $$

Therefore, using the asymptotic formula in Eq.~\eqref{eq:Gammasym}, we overall get
$$\left |\int_{\frac1{2^{n-1}}}^1 f(v)\,v^{-\beta}\,\left[(1+u(1/v-1)]^{-\beta}-P_{n-1}(u(1/v-1)\right]\,\operatorname{d}\!v\right |\le \begin{cases}
O\left(\| f\|_p\cdot \frac{n^{\beta-1}}{2^n}\right) & \text{if } 0<\beta q<1,\\
O\left(\| f\|_p\cdot \frac{n^{2\,\beta-1}}{2^n}\right) & \text{if } \beta q= 1,\\
O\left(\| f\|_p\cdot \frac{n^{\beta-1}}{2^{n/q}}\right) & \text{if } \beta q >1.\\
\end{cases}$$
Since $\beta-1<0$ and since the exponential $2^n$ grows faster than any power of $n$, this establishes the order with which 
$$ \left |\int_{\frac1{2^{n-1}}}^1 f(v)\,v^{-\beta}\,\left[(1+u(1/v-1)]^{-\beta}-P_{n-1}(u(1/v-1)\right]\,\operatorname{d}\!v\right | \ \to \ 0$$
as $n\to \infty$.



\subsection{Approximation if $u\in \left[2^{-(k+1)},2^{-k}\right]$ with $k=0,\ldots,n-1$.\label{subsec:case2}}
 
We now consider the case where $u$ in $\widetilde{(K_{1-\beta}f)}(u)$ does not fall into the leftmost interval; in this case,  we decompose $\widetilde{(K_{1-\beta}f)}(u)$ into two or three parts as follows:
\begin{equation*}
\widetilde{(K_{1-\beta}f)}(u)=\int\limits_0^1\frac{f(v)}{(u+v-uv)^{\beta}}\operatorname{d}\!v 
=
\begin{cases}
\int\limits_0^{2^{-(k+2)}} \frac{f(v)}{(u+v-uv)^{\beta}}\, \operatorname{d}\!v +\int\limits_{2^{-(k+2)}}^1  \frac{f(v)}{(u+v-uv)^{\beta}}\, \operatorname{d}\!v & \text{if } k=0,\\[4mm]
 \int\limits_0^{2^{-(k+2)}} \frac{f(v)}{(u+v-uv)^{\beta}}\, \operatorname{d}\!v +\int\limits_{2^{-(k+2)}}^1  \frac{f(v)}{(u+v-uv)^{\beta}}\, \operatorname{d}\!v & \text{if } k=1,\\[4mm]
 \begin{array}{c}\int\limits_0^{2^{-(k+2)}} \frac{f(v)}{(u+v-uv)^{\beta}}\, \operatorname{d}\!v +\int\limits_{2^{-(k+2)}}^{2^{-(k-1)}}  \frac{f(v)}{(u+v-uv)^{\beta}}\, \operatorname{d}\!v \\ +\int\limits_{2^{-(k-1)}}^1 \frac{f(v)}{(u+v-uv)^{\beta}}\, \operatorname{d}\!v \end{array}& \text {if } k\ge 2.
\end{cases}
\end{equation*}
In the following,  we will call $\int\limits_0^{2^{-(k+2)}} \frac{f(v)}{(u+v-uv)^{\beta}}\, \operatorname{d}\!v$ the \emph{first} integral, $\int\limits_{2^{-(k+2)}}^{\min\{2^{-(k-1)},1\}} \frac{f(v)}{(u+v-uv)^{\beta}}\, \operatorname{d}\!v$ the \emph{second} integral, and $\int\limits_{2^{-(k-1)}}^1 \frac{f(v)}{(u+v-uv)^{\beta}}\, \operatorname{d}\!v$ the \emph{last} integral.

We represent and approximate the last integral similar to the second integral in the case $u\in[0,2^{-n}]$, with the same order of error.

We represent the first integral, i.e.,\ $\int\limits_0^{2^{-(k+2)}} \frac{f(v)}{(u+v-uv)^{\beta}}\, \operatorname{d}\!v$, in the form
$$\int\limits_0^{2^{-(k+2)}} \frac{f(v)}{(u+v-uv)^{\beta}}\, \operatorname{d}\!v = 
\int_0^{2^{-(k+2)}}\frac{f(v)}{u^\beta \left[1+v\,\left(\frac1{u}-1\right)\right]^\beta}\,\operatorname{d}\!v.$$
We approximate $\left[1+v\,\left(\frac1{u}-1\right)\right]^{\beta}$ by partial sums of the Taylor series in Eq.~\eqref{eq:taylor1} (where we set $x=v\,\left(\frac1{u}-1\right)=\left(\frac{v}{u}-v\right)$). Then we obtain
\begin{multline*}
\left |\int_0^{2^{-(k+2)}} f(v)\,u^{-\beta}\,\left[(1+v(1/u-1)]^{-\beta}-P_{n-1}(v(1/u-1)\right]\,\operatorname{d}\!v\right | \\ \le \frac{\Gamma(\beta +n)}{ \Gamma(\beta)\Gamma(n+1)}\cdot\frac1{2^n}\cdot u^{-\beta}\,\int_0^{2^{-(k+2)}} \left | f(v)\right|\,\operatorname{d}\!v\stackrel{(\star)}{\le} O\left( \frac{n^{\beta-1}}{2^{n-\beta(k+1)}}\,\int_0^{2^{-(k+2)}} \left | f(v)\right|\,\operatorname{d}\!v \right)\le O\left(\left\| f\right\|_p\cdot \frac{n^{\beta-1}}{2^{n+\left(\frac1{q}-\beta\right)(k+1)}}\right),
\end{multline*}
where in step $(\star)$ we used that $2^{-(k+1)}\le u\le 2^{-k}$ and the asymptotic formula in Eq.~\eqref{eq:Gammasym}, while the last estimate is due to an application of Jensen's inequality in the form $\phi\left(\int_a^b |f(x)|\,\operatorname{d}\!x\right)\le \frac1{b-a}\int_a^b \phi((b-a)\cdot |f(x)|)\,\operatorname{d}\!x$  with convex function $\phi(x)=x^p$ (and $a=0$, $b=2^{-(k+1)}$) yielding overall the estimate $\int_0^{2^{-(k+2)}} \left | f(v)\right|\,\operatorname{d}\!v \le 2^{-(k+1)(1-\frac1{p})}\cdot\|f\|_p$.

Finally, we consider the second integral, $\int\limits_{2^{-(k+2)}}^{\min\{2^{-(k-1)},1\}} \frac{f(v)}{(u+v-uv)^{\beta}}\, \operatorname{d}\!v$, as a function of $u$ and approximate it by partial sums of its Taylor series of order $n$ in the neighbourhood of the point $u_k=2^{-(k+1)}+2^{-(k+2)}$. 
Note that $2^{-(k+1)}<u_k<2^{-k}$. So let,
$$F(u)=\int_{2^{-(k+2)}}^{\min \{2^{-(k-1)},1\}}f(v)(u+v-uv)^{-\beta}\,\operatorname{d}\!v.$$
We calculate the Taylor series to $F(u)$ around $u_k$:
\begin{multline*}
F(u)=F(u_k)+F'(u_k)(u-u_k)+\frac{F''(u_k)(u-u_k)^2}{2!}+\ldots \\ +\frac{F^{(n-1)}(u_k)(u-u_k)^{n-1}}{(n-1)!}+\frac{F^{(n)}(\phi)(u-u_k)^n}{n!}
\end{multline*}
where $\phi$ is a real number between $u_k$ and $u$. Here,
\begin{equation}\label{eq:poly4}
Q_{n-1}(u-u_k) = F(u_k)+F'(u_k)(u-u_k)+\frac{F''(u_k)(u-u_k)^2}{2!}+\ldots +\frac{F^{(n-1)}(u_k)(u-u_k)^{n-1}}{(n-1)!}
\end{equation}
is the Taylor polynomial of degree $(n-1)$ in $(u-u_k)$ and therefore also $u$, while the remainder term (in the Lagrangian form) is
$$\frac{F^{(n)}(\phi)(u-u_k)^n}{n!}.$$

Using differentiation under the integral sign, we can calculate the derivatives $F^{(n)}(u)$:
\begin{multline*}
F^{(n)}(u)\ 
= \ \int\limits_{2^{-(k+2)}}^{\min \{2^{-(k-1)},1\}}f(v)\left[\beta(\beta+1)\cdots(\beta+n-1)\right]\,\frac{(1-v)^n}{(u+v-uv)^{\beta +n}}\,\operatorname{d}\!v \\ = \frac{\Gamma(\beta+n)}{\Gamma(\beta)}\,  \int\limits_{2^{-(k+2)}}^{\min \{2^{-(k-1)},1\}} f(v)\,\frac{(1-v)^n}{(u+v-uv)^{\beta+n}}\,\operatorname{d}\!v.
\end{multline*}

Therefore, we get
\begin{multline*}
F(u)-Q_{n-1}(u-u_k) \ =\ \frac{F^{(n)}(\phi)(u-u_k)^n}{n!} \\
= \ \frac{\Gamma(\beta+n)}{\Gamma(\beta)\,n!}\,(u-u_k)^n\, \int\limits_{2^{-(k+2)}}^{\min \{2^{-(k-1)},1\}} f(v)\,\frac{(1-v)^n}{(\phi+v-\phi v)^{\beta+n}}\,\operatorname{d}\!v.
\end{multline*}
Note that $\frac1{2^{k+1}}< \phi< \frac1{2^k}$, and thus
\begin{multline*}
\left|\int\limits_{2^{-(k+2)}}^{\min \{2^{-(k-1)},1\}} f(v)\,\frac{(1-v)^n}{(\phi+v-\phi v)^{\beta+n}}\,\operatorname{d}\!v \right|  
\le 
\int\limits_{2^{-(k+2)}}^{\min \{2^{-(k-1)},1\}}\left|f(v)\right|\frac{(1-v)^n}{\left(\frac{1-v}{2^{k+1}}+v\right)^{n}\,\left(\frac{1-v}{2^{k+1}}+v\right)^{\beta}}\,\operatorname{d}\!v\\
\le
2^{(k+1)n}\,\int\limits_{2^{-(k+2)}}^{\min \{2^{-(k-1)},1\}}\left|f(v)\right|\frac1{\left(\frac{1-v}{2^{k+1}}+v\right)^{\beta}}\,\operatorname{d}\!v 
\le 2^{(k+1)n}\,\int\limits_{2^{-(k+2)}}^{\min \{2^{-(k-1)},1\}}\frac{\left|f(v)\right|}{v^{\beta}}\,\operatorname{d}\!v.
\end{multline*}

By H\"{o}lder's inequality, we get for this last integral in the previous line:
\begin{equation*}
\int\limits_{2^{-(k+2)}}^{\min \{2^{-(k-1)},1\}}\frac{\left|f(v)\right|}{v^{\beta}}\,\operatorname{d}\!v \le  \|f\|_p\, \left(\int\limits_{2^{-(k+2)}}^{\min \{2^{-(k-1)},1\}}v^{-\beta q}\,\operatorname{d}\!v\right)^{\frac1{q}} 
 =  \begin{cases}
\|f\|_p\,\left[\left[\ln v\right ]_{v=2^{-(k+2)}}^{v=\min \{2^{-(k-1)},1\}}\right ]^{\frac1{q}} & \text{if } \beta q= 1,\\
\|f\|_p\,\left[\left[\frac{v^{1-\beta q}}{1-\beta q}\right ]_{v=2^{-(k+2)}}^{v=\min \{2^{-(k-1)},1\}}\right ]^{\frac1{q}} & \text{if } \beta q\neq 1.
\end{cases} 
\end{equation*}
We consider cases:
\begin{itemize}
\item If $\beta q=1$, then 
\begin{equation*}
\left[ \left[\ln v\right ]_{v=2^{-(k+2)}}^{v=\min \{2^{-(k-1)},1\}}\right ]^{\frac1{q}} = \left\{ \begin{array}{ll} (2\ln 2)^{\beta} & \text{if } k=0,\\ (3\ln 2)^{\beta} & \text{if } k>0,\end{array} \right\} \le (3\ln 2)^{\beta}.
\end{equation*}

\item If  $\beta q\neq1$, then we have for $k=0$ that 
\begin{equation*}
\left[\left[\frac{v^{1-\beta q}}{1-\beta q}\right ]_{v=2^{-(k+2)}}^{v=\min \{2^{-(k-1)},1\}}\right ]^{\frac1{q}}  =  \left(\frac{1-4^{\beta q-1}}{1-\beta q} \right)^{\frac1{q}} < \begin{cases} 
 \frac{1}{(1-\beta q)^{1/q}} & \text{if } 0<\beta q<1,\\
 \frac{4^{\beta-1/q}}{(\beta q-1)^{1/q}} & \text{if } \beta q>1,\\
\end{cases}
\end{equation*}
while for $k>0$ we have
\begin{multline*}
\left[\left[\frac{v^{1-\beta q}}{1-\beta q}\right ]_{v=2^{-(k+2)}}^{v=\min \{2^{-(k-1)},1\}}\right ]^{\frac1{q}} \ 
= \ \left[\frac{2^{3(1-\beta q)}-1}{(1-\beta q)2^{(k+2)(1-\beta q)}}\right ]^{1/q}\\
\le  \  \begin{cases} 
\frac{7^{1/q}}{(1-\beta q)^{1/q}\,2^{(k+2)(1/q-\beta)}} & \text{if } 0<\beta q<1,\\[4mm]
\frac{2^{(k+2)(\beta-1/q)}}{(\beta q-1)^{1/q}} & \text{if } \beta q>1.
\end{cases}
\end{multline*}
In either case, we obtain
$$ \left[\left[\frac{v^{1-\beta q}}{1-\beta q}\right ]_{v=2^{-(k+2)}}^{v=\min \{2^{-(k-1)},1\}}\right ]^{\frac1{q}}  \le \begin{cases} 
\frac{7^{1/q}}{(1-\beta q)^{1/q}\,2^{(k+2)(1/q-\beta)}} & \text{if } 0<\beta q<1,\\[4mm]
\frac{2^{(k+2)(\beta-1/q)}}{(\beta q-1)^{1/q}} & \text{if } \beta q>1,\\
\end{cases} $$
since $4^{\beta}>1$ and $\left(\frac74\right)^{1/q}>1$.
\end{itemize}

In total we now get (in the step $(\star)$ we use that $|u-u_k|\le 2^{-k}-2^{-(k+1)}-2^{-(k+2)}=2^{-(k+2)}$)
\begin{multline*}
\left| F(u)-Q_{n-1}(u-u_k) \right|   
=  \frac{\Gamma(\beta+n)}{\Gamma(\beta)\,\Gamma(n+1)}\,\left|u-u_k\right|^n\, \left|\int\limits_{2^{-(k+2)}}^{\min \{2^{-(k-1)},1\}} f(v)\,\frac{(1-v)^n}{(\phi+v-\phi v)^{\beta+n}}\,\operatorname{d}\!v \right| \\
\stackrel{(\star)}{\le}  \frac{\Gamma(\beta+n)}{\Gamma(\beta)\,\Gamma(n+1)}\,2^{-(k+2)n}\, \left|\int\limits_{2^{-(k+2)}}^{\min \{2^{-(k-1)},1\}} f(v)\,\frac{(1-v)^n}{(\phi+v-\phi v)^{\beta+n}}\,\operatorname{d}\!v \right| \\
\le \frac{\Gamma(\beta+n)}{\Gamma(\beta)\,\Gamma(n+1)}\,2^{-(k+2)n}\,2^{(k+1)n}\,\int\limits_{2^{-(k+2)}}^{\min \{2^{-(k-1)},1\}}\frac{\left|f(v)\right|}{v^{\beta}}\,\operatorname{d}\!v \\
\le\frac{\Gamma(\beta+n)}{\Gamma(\beta)\,\Gamma(n+1)}\,2^{-n}\, \|f\|_p\,\left(\int\limits_{2^{-(k+2)}}^{\min \{2^{-(k-1)},1\}}v^{-\beta q}\,\operatorname{d}\!v\right)^{\frac1{q}}\\
\le \begin{cases}
O\left(\|f\|_p\cdot {{}^{n^{\beta-1}}/_{2^{(k+2)(1/q-\beta)+n}}} \right) & \text{if } 0<\beta q<1,\\
O\left(\|f\|_p\cdot {}^{n^{\beta-1}}/_{2^{n}} \right) & \text{if } \beta q=1,\\
O\left(\|f\|_p\cdot {}^{{n^{\beta-1}}\cdot 2^{\beta(k+2)}}/_{2^{\frac1{q}(k+2)+n+1}} \right) & \text{if } \beta q>1,\\
\end{cases}
\end{multline*}
where in the last step we again made use of the asymptotic formula in Eq.~\eqref{eq:Gammasym}.

With the above reasoning and observing  that $(\frac1{q}-\beta)>0$ if $0<\beta q<1$, that $(k+2)\le (n+1)$ and $\beta<1+\frac1{q}$, this establishes the order with which 
$$ \left|\int\limits_{2^{-(k+2)}}^{\min \{2^{-(k-1)},1\}}f(v)(u+v-uv)^{-\beta}\,\operatorname{d}\!v-Q_{n-1}(u-u_k) \right|\ \to \ 0$$
as $n\to \infty$.

\subsection{Error of approximation, case $0<\beta q <1$.}\label{subsec:error1}

If $0<\beta q<1$, we approximate  $\widetilde{(K_{1-\beta}f)}$ using the $L^{\infty}$-norm. We have found that
\begin{itemize}
\item  for $u\in[0,2^{-n}]$, we can approximate  $\widetilde{(K_{1-\beta}f)}$ by a polynomial of order $n$ with order of error
\begin{multline*}
O\left(\|f\|_p\cdot 2^{-n\left(\frac1{q}-\beta\right)}\right)+O\left(\|f\|_p\cdot \frac{n^{\beta-1}}{2^n}\right) \\
= O\left(\|f\|_p\cdot 2^{-n\left(\frac1{q}-\beta\right)}+\|f\|_p\cdot 2^{-n-(1-\beta)\log_2 n}\right) 
= O\left(\|f\|_p\cdot 2^{-\kappa n}\right). 
\end{multline*}
where 
$$ \kappa < \min\left\{\frac1{q}-\beta,1\right\} = \frac1{q}-\beta=\alpha-\frac1{p}.$$
\item for $u\in[2^{-(k+1)},2^{-k}]$, $k=0,\ldots ,n-1$(noting that $k+1\le n$) we can approximate  $\widetilde{(K_{1-\beta}f)}$ with order of error
\begin{multline*}
O\left(\|f\|_p\cdot\frac{n^{\beta-1}}{2^{n+(\frac1{q}-\beta)(k+1)}}\right)+O\left(\|f\|_p\cdot {{}^{n^{\beta-1}}/_{2^{(k+2)(1/q-\beta)+n}}} \right) +O\left(\|f\|_p\cdot\frac{n^{\beta-1}}{2^n}\right)\\ \le 
O\left(\|f\|_p\cdot 2^{-n\left(1+\frac1{q}-\beta\right)-(1-\beta)\log_2 n}
+\|f\|_p\cdot 2^{-n-(1-\beta)\log_2 n}\right) 
= O\left(\|f\|_p\cdot 2^{-\hat{\kappa} n}\right)
\end{multline*}
where 
$$\hat{\kappa}<\min\left\{1+\frac1{q}-\beta,1\right\} = 1+\frac1{q}-\beta = \alpha+\frac1{q}=1+\alpha-\frac1{p}.$$
\end{itemize}
Overall, we approximate  $\widetilde{(K_{1-\beta}f)}$, and thus $(K_{\alpha}f)$, with order of error 
\begin{equation}\label{eq:app1}
O\left(\|f\|_p\cdot 2^{-\kappa n}\right)
\end{equation} 
for any positive constant $\kappa<\alpha-\frac1{p}$. 

\subsection{Error of approximation, case $\beta q = 1$.}\label{subsec:error2}

Because the approximation for $u\in[0,2^{-n}]$ has an error of order $O\left(\|f\|_p\cdot \left( \frac{\ln 3-\ln\left({2^{n}\,u}\right)}{1-2^{-n}} \right)^{1/q}\right)$ and $\lim\limits_{u\to 0^+} -\ln(2^n u) = +\infty$, the operator  $\widetilde{(K_{1-\beta}f)}$ cannot be approximated in $L^{\infty}[0,1]$, only in $L^r[0,1]$ for $1\le r<\infty$.

Since $x\mapsto x^r$ with $r>1$ is a convex function, we will use Jensen's inequality in the form
$(x+y)^r \le 2^{r-1}\,\left(x^r+y^r\right)$ respectively $(x+y+z)^r \le 3^{r-1}\,\left(x^r+y^r+z^r\right)$
in the following estimate. While we have calculated the order of the error of our approximation of  $\widetilde{(K_{1-\beta}f)}$ pointwise above, for a function $f\in L^r[0,1]$ we have
\begin{equation*}
\|f\|_r =  \left(\int_0^1 |f(x)|^r\,\operatorname{d}\!x\right)^{1/r}  =  \left(\int_0^{2^{-n}} |f(x)|^r\,\operatorname{d}\!x+\sum_{k=0}^{n-1} \int_{2^{-(k+1)}}^{2^{-k}} |f(x)|^r\,\operatorname{d}\!x\right)^{1/r}
\end{equation*}
(and we call $\int_0^{2^{-n}} |f(x)|^r\,\operatorname{d}\!x$ respectively $\int_{2^{-(k+1)}}^{2^{-k}} |f(x)|^r\,\operatorname{d}\!x$ the contributions of the intervals to the $L^r$-norm). The contributions to the order of the error are as follows:
\begin{itemize}
\item For $u\in[0,2^{-n}]$, we find that the contribution to the order of the error is less than or equal to 
\begin{equation*}
O\left(\|f\|^r_p\cdot\left[\int_0^{2^{-n}} \left( \frac{\ln 3-\ln\left({2^{n}\,u}\right)}{1-2^{-n}} \right)^{r/q}\, \operatorname{d}\!u\right]+\|f\|^r_p\cdot\left[\frac{n^{r\,(2\,\beta-1)}}{2^{r\,n}} \right] \cdot 2^{-n} \right). 
\end{equation*}
For the second term, note that due to the length of the interval being $2^{-n}$ and since the term in square brackets goes to zero as $n$ goes to infinity, its order is always less than $O(2^{-n})$.\\
We use the following formula for the (improper) integral (assuming $a,b>0$):
\begin{multline*}
\int_0^1 (a-b\,\ln x)^{\gamma}\,\operatorname{d}\!x \ 
= \ a^{\gamma} + b\,\int_0^1 (a-b\,\ln x)^{\gamma-1}\,\operatorname{d}\!x \\  \stackrel{(\star)}{=} \ a^{\gamma} + b\,a^{\gamma-1}+b^2\,a^{\gamma-2}+b^3\,a^{\gamma-3}+\ldots 
\le 
\frac{a^{\gamma+1}}{a-b},
\end{multline*}
where we note that the sum after $(\star)$ is a finite one if $r$ is a natural number (however, the expression $a^{\gamma+1}/(a-b)$ is still an upper bound in that case), and the sum converges if $\frac{b}{a}<1$ , i.e., $b<a$.  Using the substitutions $x=2^n \,u$ (thus, $\operatorname{d}\!u=2^{-n}\,\operatorname{d}\!x$), $a=\ln 3/(1-2^{-n})$, $b=1/(1-2^{-n})$ and $\gamma=r/q$, we therefore get (note that $a>b$ here)
\begin{multline*}
\int_0^{2^{-n}} \left( \frac{\ln 3-\ln\left({2^{n}\,u}\right)}{1-2^{-n}} \right)^{r/q}\, \operatorname{d}\!u  =  2^{-n}\,\int_0^1 \left(\frac{\ln 3}{1-2^{-n}} - \frac{1}{1-2^{-n}} \ln x\right)^{r/q}\,\operatorname{d}\!x \\ 
\le  2^{-n}\,\frac{(\ln 3)^{\frac{r}{q}+1}}{\ln 3 - 1} \cdot \frac{1-2^{-n}}{\left(1-2^{-n}\right)^{\frac{r}{q}+1}} 
\ \le\  
2^{-n}\,\frac{(2\,\ln 3)^{\frac{r}{q}+1}}{\ln 3 - 1}.
\end{multline*} 
Overall, the contribution for the interval $[0,2^{-n}]$ to the order of error is thus less than or equal to $O(2^{-n})$.
\item For $u\in\left[2^{-(k+1)},2^{-k}\right]$, $k=0,\ldots,n-1$, we have that the contribution to the order of the error is less than or equal to (note that $2^{-(k+1)}$ is the interval length)
\begin{multline*}
O\left( \left[\left[\frac{n^{\beta-1}}{2^{n+(\frac1{q}-\beta)(k+1)}}\right]^r +  \left[\frac{n^{\beta-1}}{2^{n}}\right]^r + \left[\frac{k^{\beta}\,n^{\beta-1}}{2^n} \right]^r\right]\cdot 2^{-(k+1)}\cdot\|f\|^r_p \right) \\
\le O\left( \left[2^{-r\,(n+(1-\beta)\,\log_2n)} + 2^{-r\,(n+(1-\beta)\,\log_2n)}  +  2^{-r\,(n+1-2\beta)\,\log_2n)}  \right]\cdot 2^{-(k+1)} \cdot\|f\|^r_p\right) \\
\le O\left(2^{-r\,\hat{\kappa}\,n}\cdot 2^{-(k+1)}\cdot\|f\|^r_p \right),
\end{multline*}
for any positive constant $\hat{\kappa}<1$.
\end{itemize}
Overall, we approximate $\widetilde{(K_{1-\beta}f)}$ in $L^{r}[0,1]$ with order of error
\begin{multline}\label{eq:app2}
O\left(\left[2^{-n}+2^{-r\,\hat{\kappa}\,n}\sum_{k=0}^{n-1} 2^{-(k+1)}\right]^{1/r}\cdot\|f\|_p\right) =O\left(\left[2^{-n}+2^{-r\,\hat{\kappa}\,n}\left(1- 2^{-n}\right)\right]^{1/r}\cdot\|f\|_p\right) \\ =O\left(\left[2^{-n}+2^{-r\,\hat{\kappa}\,n}\right]^{1/r}\cdot\|f\|_p\right) \le O\left(2^{-\kappa\,n}\cdot\|f\|_p\right)
\end{multline}
(in the last step we used that $x\mapsto x^{1/r}$ is a strictly increasing function) for any positive constant $\kappa<\min\{\frac1{r},1\}$. Note that this holds for any $1\le r<\infty$ -- so the approximation might in general just fail to be in $L^{\infty}[0,1]$ but belongs to any other $L^r[0,1]$.

\subsection{Error of approximation, case $\beta q >1$.}\label{subsec:error3}

 Because the approximation for $u\in[0,2^{-n}]$ has an error of order $O\left(2^{-n/q}\,u^{-\beta}\cdot\|f\|_p\right)$ and $\lim\limits_{u\to 0^+} u^{-\beta} = +\infty$, the operator $\widetilde{(K_{1-\beta}f)}$ cannot be approximated in $L^{\infty}[0,1]$, only in $L^r[0,1]$ for some appropriate $1\le r<\infty$.

As before, we will use Jensen's inequality to estimate the order of the error in $L^r[0,1]$-norm form our pointwise estimates. Here, we get the following contributions:
\begin{itemize}
\item  For $u\in[0,2^{-n}]$, we have that the contribution to the order of the error is less than or equal to  
$$ O\left(\left[2^{-n\,r/q} \int_0^{2^{-n}} u^{-r\,\beta}\, \operatorname{d}\!u\right]\cdot\|f\|^r_p+\left[\frac{n^{-r(1-\beta)}}{2^{r\,n/q}} \right] \cdot 2^{-n} \cdot\|f\|^r_p\right). $$
Here we note that the (improper) integral only exists if $(-r\,\beta)>-1$ in which case we get
$$ \int_0^{2^{-n}} u^{-r\,\beta}\, \operatorname{d}\!u \ =\ \left.\frac{u^{1-r\,\beta}}{1-r\,\beta}\right|_{0}^{2^{-n}}=\frac1{1-r\,\beta}\,2^{-n(1-r\,\beta)}.$$
Thus the total contribution is
\begin{equation*}
 O\left(2^{-n\left(1-r\,\left[\beta-\frac1{q}\right]\right)}\cdot\|f\|^r_p+2^{-n\left(1+\frac{r}{q}\right)-r(1-\beta)\log_2 n}\cdot\|f\|^r_p\right) \le\ O\left(2^{-\widetilde{\kappa} n}\cdot\|f\|^r_p\right) 
\end{equation*}
for any positive number 
$$\widetilde{\kappa}<\min\left\{1-r\beta+\frac{r}{q},1+\frac{r}{q}\right\}=1-r\beta+\frac{r}{q} = 1-r\left(\frac1{p}-\alpha\right).$$ 
Recall that $\alpha<\frac1{p}$ and that $r$ has to be chosen such that $r\beta<1$, i.e., $r<1/(1-\alpha)$ respectively $\alpha>1-\frac1{r}$. Therefore, 
$$1-r\left(\frac1{p}-\alpha\right) > 1-\frac{1-p\alpha}{p(1-\alpha)} = \frac{p-1}{p(1-\alpha)} \ge 0 $$
since $0<\alpha<1$ and $p\ge 1$. So, $\widetilde{\kappa}$ is indeed bounded by a positive number.
\item For $u\in\left[2^{-(k+1)},2^{-k}\right]$, $k=0,\ldots,n-1$, we have that the contribution to the order of the error is less than or equal to (note that $2^{-(k+1)}$ is the interval length)
\begin{multline*}
O\left( \left[\left[\frac{n^{\beta-1}}{2^{n+(\frac1{q}-\beta)(k+1)}}\right]^r +  \left[\frac{{n^{\beta-1}}\cdot 2^{\beta(k+2)}}{2^{\frac1{q}(k+2)+n+1}} \right]^r + \left[\frac{n^{\beta-1}}{2^{n-k}\,2^{k/q}} \right]^r\right]\cdot 2^{-(k+1)}\cdot\|f\|^r_p \right) \\
\le O\left(2^{-r\hat{\kappa}n}\cdot 2^{-(k+1)}\cdot\|f\|^r_p\right),
\end{multline*}
for any positive constant $\hat{\kappa}<\min\left\{1+\frac1{q}-\beta,\frac1{q}\right\}=\frac1{q}=1-\frac1{p}$.
\end{itemize}
Overall, we approximate $\widetilde{(K_{1-\beta}f)}$ respectively $(K_{\alpha}f)$ in $L^r[0,1]$ where $r<\frac1{1-\alpha}$  (respectively $\alpha>1-\frac1{r}$) with order of error
\begin{multline}\label{eq:app3}
O\left(\left[2^{-\widetilde{\kappa}\,n}+2^{-r\,\hat{\kappa}\,n}\sum_{k=0}^{n-1} 2^{-(k+1)}\right]^{1/r}\cdot\|f\|_p\right)  =O\left(\left[2^{-\widetilde{\kappa}\,n}+2^{-r\,\hat{\kappa}\,n}\left(1- 2^{-n}\right)\right]^{1/r}\cdot\|f\|_p\right) \\ =O\left(\left[2^{-\widetilde{\kappa}n}+2^{-r\,\hat{\kappa}\,n}\right]^{1/r}\cdot\|f\|_p\right) \ \le\ O\left(2^{-\kappa\,n}\cdot\|f\|_p\right)
\end{multline}
for any positive constant 
$$\kappa<\min\left\{\frac1{r}-\frac1{p}+\alpha,1-\frac1{p}\right\}=1-\frac1{p}.$$ 
Here we note that $\frac1{r}-\frac1{p}+\alpha>1-\frac1{p}$ since $r<\frac1{1-\alpha}$. 

\subsection{Dimension of the approximation.}\label{subsec:dim}

Now that we have obtained the order of the error of the approximation, we have to determine the dimension of the subspace of $L^r[0,1]$ respectively $L^{\infty}[0,1]$ that we use in this approximation.
\begin{itemize}
\item For $u\in[0,2^{-n}]$, the approximation is a polynomial of degree $(n-1)$ in $u$, see the second integral in Section~\ref{subsec:case1}. Thus it has the form 
$$c_0+c_1\,u+\ldots+c_{n-1}\,u^{n-1},$$
for some constants $c_i\in\mathbb{R}$. Obviously, the subspace has dimension $n$.
\item Let $k=0,1,\ldots, n-1$. For $u\in(2^{-(k+1)},2^{-k}]$ the approximation obtained has the form  
$$a^{(k)}_0+a^{(k)}_1\,u+\ldots+a^{(k)}_{n-1}\,u^{n-1}+\frac{b^{(k)}_0}{u^{\beta}}+ \frac{b^{(k)}_1}{u^{\beta+1}}+\ldots+\frac{b^{(k)}_{n-1}}{u^{\beta+n-1}},$$
for some constants $a^{(k)}_i, b^{(k)}_i\in\mathbb{R}$; here, the polynomial coefficients $a^{(k)}_i$ come from the approximation by the second and the last integral in Section~\ref{subsec:case2}, while the coefficients $b^{(k)}_i$ come from the approximation by the first integral in Section~\ref{subsec:case2}. Noting that $0<\beta<1$, the subspace in question here has dimension $2n$.
\end{itemize}

Since on each of the $n$ intervals $(2^{-(k+1)},2^{-k}]$, $k=0,1,\ldots, n-1$, the subspace we used for the approximation has dimension $(2n)$, and for the interval $[0,2^{-n}]$ the subspace has dimension $n$, our ``piecewise-smooth approximation'' has dimension $n+n\cdot (2n)=2n^2+n$ (also compare to \cite[p.~29]{BS}).


Thus taking into account the dimension of the subspaces and using the original operator $K_{\alpha}$ again (and noting that $\alpha+\beta=1$ and $\frac1{p}+\frac1{q}=1$), we obtain from Equations~\eqref{eq:app1}, \eqref{eq:app2} and \eqref{eq:app3} that
\begin{eqnarray*}
d_{2n^2+n}\left(K_{\alpha}: L^p[0,1]\to L^{\infty}[0,1]  \right) & \le\ O\left( 2^{-\kappa_1\, n}\cdot\|f\|_p \right) & \text{if }\frac1{p}<\alpha<1, \\
d_{2n^2+n}\left(K_{\alpha}: L^p[0,1]\to L^{r}[0,1]  \right) & \le\ O\left( 2^{-\kappa_2\, n} \cdot\|f\|_p\right) & \text{if }\alpha=\frac1{p},\ 1\le r<\infty, \\
d_{2n^2+n}\left(K_{\alpha}: L^p[0,1]\to L^{r}[0,1]  \right) & \le\ O\left( 2^{-\kappa_3\,n}\cdot\|f\|_p \right) & \text{if }1-\frac1{r}<\alpha<\frac1{p}, 
\end{eqnarray*}
for some positive constants $\kappa_1=\kappa_1(\alpha,p)$, $\kappa_2=\kappa_2(r)$ and $\kappa_3=\kappa_3(p)$.

Finally, noting that $m=2n^2+n$ for $m,n>0$ implies
\begin{equation*}
n=\frac14\left(\sqrt{8m+1}-1\right) \le \frac{\sqrt{m}}{\sqrt{2}}-\frac14+\frac1{16}\,\frac1{\sqrt{2m}} 
\end{equation*}
by Bernoulli's inequality, we have now established the main theorem.\qed

Our proof here parallels the proof in \cite[Section 2]{BL} with the necessary changes and adding a few more details (e.g., by carefully considering the cases $k=0$, $k=1$ and $k\ge 2$ in Section~\ref{subsec:case2}). Given the similarity of the integral operators $S_{\alpha}$ (considered in \cite{BL}) and $K_{\alpha}$, it is maybe not a big surprise that the results are similar, especially since we used similar methods to establish them. Informally, we can justify the use of this method for $K_{\alpha}$ as follows: Since the kernel of $\widetilde{K_{1-\beta}}$ has a singularity at $(0,0)$, the approximation will be ``worst'' near $0$. Thus, as the order of approximation increases, the goal is to make this part of ``bad'' approximation near $0$ in such a way smaller that the contributions to the Kolmogorov widths from this part near $0$ and the remaining part are of the same order; this is here achieved by considering the interval $[0,2^{-n}]$ and the remaining intervals. The example in Section~\ref{sec:example}, also see Fig.~\ref{fig:case1n}, should make this remark clearer.

\section{Remark on Entropy Numbers}\label{sec:entropy}

While the Kolmogorov widths of a set $A$ characterize the error of approximation of $A$ by $n$-dimensional subspaces, the notion of \emph{metric entropy} -- also introduced by Kolmogorov, see \cite{Kol56,Kol58, KT59, Vit61} -- characterises how well one can approximate a compact set $A$ by finite sets: For a given set $A\subset X$ in a metric space $X$, a family $U_1, U_2,\ldots$ of subsets of $X$ is an \emph{$\varepsilon$-covering of $A$} if the radius of each $U_k$ does not exceed $\varepsilon$ and if the sets $U_k$ cover $A$. Obviously, for a given $\varepsilon>0$ and compact set $A$, a finite number of such sets $U_k$ suffices to cover $A$; we denote the minimal number of sets of radius $\varepsilon$ that cover $A$ by $N_{\varepsilon}(A)$. The logarithm
\begin{equation*}
H_{\varepsilon}(A) = \log N_{\varepsilon}(A)
\end{equation*}
is called the \emph{metric entropy} (or \emph{$\varepsilon$-entropy}) \emph{of the set $A$ in $X$}. One can restate this definition by saying that $N_{\varepsilon}(A)$ is the number of points in a \emph{minimal $\varepsilon$-net}; we also note that there is a closely related concept of \emph{$\varepsilon$-capacity} $C_{\varepsilon}(A) = \log M_{\varepsilon}(A)$ where $M_{\varepsilon}(A)$ denotes the number of points in a \emph{maximal $\varepsilon$-distinguishable set}, see \cite[Section 15.1]{LGM}. These two concepts are related by $C_{2\varepsilon}(A)\le H_{\varepsilon}(A)\le C_{\varepsilon}(A)$, see \cite[Theorem 10.1.1]{L} and \cite[Proposition 15.1.1]{LGM}.

We define \emph{(dyadic) entropy numbers} $e_n(A)$ of a set $A$ in a metric space $X$ by\footnote
{
	If the index $n$ starts with $1$, the definition
	\begin{equation*}
	e_n(A) = \inf\left\{\varepsilon\mathbin: \text{there exist}\ 2^{n-1}\ \text{closed balls in}\ X\ \text{of radius}\ \varepsilon\ 	
	\text{covering}\ A\right\}
	\end{equation*}
	is usually used in the literature.
}
\begin{equation*}
e_n(A) = \inf\left\{\varepsilon\mathbin: \text{there exist}\ 2^n\ \text{closed balls in}\ X\ \text{of radius}\ \varepsilon\ \text{covering}\ A\right\}.
\end{equation*}
In some sense,  entropy numbers are the inverse function to $H_{\varepsilon}(A)$; it follows directly from the definition that $e_n(A)=\varepsilon$ is equivalent to $H_{\varepsilon^+}(A) \le n \log 2 = \log 2^n < H_{\varepsilon}(A)$. As with Kolmogorov widths, we define entropy numbers of a linear operator $T:X\to Y$ acting between two Banach spaces $X$ and $Y$ by $e_n(T: X\to Y)=e_n(T(B_X))$ where $B_X$ denotes the unit ball of $X$; in other words, 
\begin{equation*}
e_n(T: X\to Y) = \inf\left\{\varepsilon\mathbin: \text{there is an}\ \varepsilon\text{-net for}\ T(B_X)\ \text{in}\ Y\ \text{consisting of}\ 2^n\ \text{elements}\right\}.
\end{equation*}
Properties of entropy numbers as well as their relation to approximation numbers like Kolmogorov widths can be found, e.g., in  \cite{CS90}, \cite{ET}, \cite[Chapter 12]{Pietsch} and \cite[Chapter 5]{Pisier}.

For a good estimate of the entropy numbers $e_n(K)$ one needs not just the Kolmogorov width $d_n(K)$, but the whole sequence $d_0(K),\ldots, d_n(K)$, compare \cite{Carl81} and \cite[Sections 15.4 \& 15.7]{LGM}. In fact, we use our upper estimate on the Kolmogorov widths to obtain a (lower) estimate on the metric entropy $H_{\varepsilon}$ by \cite[Theorem 2]{Lor} (also see \cite[Theorem 15.3.2]{LGM}) from which then an upper estimate for the entropy numbers follows via the remark above. As in \cite[Corollary 3.4]{BL2} and \cite[Corollaries 1.2 \& 3.4]{BL}, we obtain from the upper bound $d_n\le O(2^{-\kappa\sqrt{n}})$ the following estimate for the entropy numbers:

\begin{thm}
Let $K_{\alpha}$ be the integral operator 
$$(K_{\alpha}f)(x)= \int_0^1 (1-xy)^{\alpha -1}\,f(y)\,\operatorname{d}\!y\qquad\text{where}\quad 0<\alpha <1.$$ 
Then the entropy numbers of $K_{\alpha}$ are asymptotically bounded as follows:
\begin{eqnarray*}
e_{n}\left(K_{\alpha}: L^p[0,1]\to L^{\infty}[0,1]  \right) & \le\ O\left( 2^{-c_1 \sqrt[3]{n}} \right) & \text{if }\frac1{p}<\alpha<1, \\
e_{n}\left(K_{\alpha}: L^p[0,1]\to L^{r}[0,1]  \right) & \le\ O\left( 2^{-c_2 \sqrt[3]{n}} \right) & \text{if }\alpha=\frac1{p},\ 1\le r<\infty, \\
e_{n}\left(K_{\alpha}: L^p[0,1]\to L^{r}[0,1]  \right) & \le\ O\left( 2^{-c_3 \sqrt[3]{n}} \right) & \text{if }1-\frac1{r}<\alpha<\frac1{p},
\end{eqnarray*}
for some positive constants $c_1$, $c_2$ and $c_3$.
\end{thm} 

\begin{proof}
Using the notation of \cite{Lor} and \cite[Section 15.3]{LGM}, let $\delta_n=2^{-\kappa\sqrt{n}}$ and denote the sequence of this numbers by $\Delta=\{\delta_0,\delta_1,\ldots\}$. Then, we have $N_i=\min\left\{k\mathbin:\delta_k\le e^{-i}\right\} = \lceil \frac{i^2}{(\kappa \log 2)^2} \rceil\approx \frac1{(\kappa \log 2)^2}\cdot i^2$. Given our linear integral operators $K_{\alpha}: L^p[0,1]\to L^r[0,1]$, we approximate $K_{\alpha}f$ by a finite dimensional subspace of $L^r[0,1]$ using the linearly independent functions $1, u,  u^{-\beta}, u^2,  u^{-\beta-1}, \ldots$, $u^{n-1}, u^{-\beta-n+1}$, compare Section~\ref{subsec:dim}. We denote the subspace spanned by the first $n$ by $Y_n$ (with $Y_0=\{0\}$), and define the distance from this $n$-dimensional subspace by $E_n(f)=\inf\{\|f-g\|\mathbin: g\in Y_n\}$. We define the \emph{approximation set} by
\begin{equation*}
A(\Delta) = \left\{f\in L^r[0,1]\mathbin: E_n(f)\le \delta_n,\ n=0,1,2,\ldots \right\};
\end{equation*} 
i.e., $A(\Delta)$ contains all elements in $L^r[0,1]$ that are not further than $\delta_n$ from the approximating subspace away (using the terminomology of \cite[Chapter 7]{DL}, the set $A(\Delta)$ is the \emph{ball of an approximation space}). In particular, we have that  $K_{\alpha}(B_{L^p[0,1]}) \subset A(\Delta)$, i.e., the set $A(\Delta)$ contains the image of the unit ball in $L^p[0,1]$ under $K_{\alpha}$. Geometrically, we might think of $A(\Delta)$ as a set obtained from a ball intersected with cylinder sets (along subspaces), which contains the ellipsoid $K_{\alpha}(B_{L^p[0,1]})$; \cite[Theorem 2]{Lor} (respectively, \cite[Theorem 15.3.2]{LGM}) then states how this set $A(\Delta)$ can be covered by balls of radius $\varepsilon$ in $L^r[0,1]$.

Since the numbers $\delta_n$ decrease to zero rapidly, we can apply \cite[Theorem 15.3.3(ii)]{LGM} -- a consequence of \cite[Theorem 2]{Lor} (also see \cite[Theorem 15.3.2]{LGM}) -- which states: For a given $0<\varepsilon<1$, let $j$ be defined by $e^{-(j-1)}<\varepsilon\le e^{-(j-2)}$, i.e., $j=\lfloor 2-\log\varepsilon\rfloor$. If $\delta_{\rho n}/\delta_n\to 0$ as $n\to\infty$ for each $\rho>1$, then $C_{\varepsilon}(A(\Delta)) \approx H_{\varepsilon}(A(\Delta))\approx N_1+\ldots+N_j$.

In our case, we have $\delta_{\rho n}/\delta_n = 2^{-\kappa (\sqrt{\rho}-1) \sqrt{n}}\to 0$ as $n\to\infty$ since $\kappa>0$ and $\sqrt{\rho}-1>0$, and $N_1+\ldots+N_j \approx \frac1{(\kappa \log 2)^2} \sum_{\ell=1}^j \ell^2 = \frac{j(j+1)(2j+1)}{6\,(\kappa \log 2)^2}$.  Thus, the asymptotic growth of the metric entropy $H_{\varepsilon}(A(\Delta))$ is cubic in $j=\lfloor 2-\log\varepsilon\rfloor$, in other words, $H_{2^{-n}}(A(\Delta))$ grows as $C\cdot n^3$ for some constant $C$, and thus the entropy number $e_{C\cdot n^3}$ as $2^{-n}$, which yields $e_n(A(\Delta)) \approx 2^{-c \sqrt[3]{n}}$ for some constant $c$.

We now repeat these calculations with the Kolmogorov widths $d_n$ in place of their upper bounds $\delta_n$. Then, the sets $N_i$ grow at least quadratically in $i$, and therefore the metric entropy $H_{\varepsilon}$ has at least cubic growth while the entropy numbers $e_n$ decrease to zero of order $2^{-c \sqrt[3]{n}}$ or faster. This establishes the claim.
\end{proof}

Generalizing the previous proof, we have actually shown that for a compact operator $K$ 
\begin{equation*}
d_n(K) \le O(\exp(-\kappa\,n^{1/q})) \quad \text{implies} \quad e_n(K) \le O(\exp(-c\, n^{1/(q+1)}))
\end{equation*}
for $q>0$ and some constants $\kappa,c>0$. We contrast this with \cite[Theorem 3.1]{BL2} which states that if  $a>0$ and $b\in\mathbb{R}$, then
\begin{equation*}
d_n(K) \le O(n^{-a} (\log n)^b) \quad \text{implies} \quad e_n(K) \le O(n^{-a} (\log n)^b),
\end{equation*}
and general estimates like $d_n(K)\le n\cot e_n(K)$ (\cite[Theorem 12.3.2]{Pietsch}), or a so-called Jackson-type inequality for an operator $K$ acting between Hilbert spaces reading $d_n(K) \le 2\, e_n(K)$ (see \cite[Formulae 2.2.12 \& 3.0.9]{CS90}). 

In Proposition~\ref{prop:KSchatten}, we established that $K_{\alpha}: L^2[0,1]\to L^2[0,1]$ belongs to the Schatten class $S_p$ for every $0<p<\infty$; in particular, the sequence of singular values $\lambda_n$ is an $\ell^p$-sequence. By our upper bound, the entropy numbers are rapidly decreasing and thus also an $\ell^p$-sequence for any $p$, in accordance with the statement of \cite[Theorem 15.7.3]{LGM}: Given a compact linear operator $K: X\to X$ on a Hilbert space $X$ and $p>0$, then $(e_n(K))\in\ell^p$ iff $(\lambda_n)\in\ell^p$.

\section{Example}\label{sec:example}

We now look at the integral transformation of the following family of functions:
$$  f_{\nu\mu}(x) = (1-x)^{\nu-1}\,x^{\mu-1}, \qquad\text{i.e.,}\quad \widetilde{f_{\nu\mu}}(u) = u^{\nu-1}\,(1-u)^{\mu-1}.$$
Note that the function $\widetilde{f_{\nu\mu}}$ has a singularity at $u=0$ if $\nu<1$ and a singularity at $u=1$ if $\mu<1$. Furthermore, $\widetilde{f_{\nu\mu}}\in C[0,1]$ if $\nu\ge 1$ and $\mu\ge 1$, and $\widetilde{f_{\nu\mu}}\not\in L^1[0,1]$ if either $\nu\le 0$ or $\mu\le 0$. Thus, for our integral operator the cases where at least one of $\nu$ or $\mu$ is between $0$ and $1$ are interesting.

Using~\cite[Formula 3.179(8)]{GR}, we get 
\begin{equation}\label{eq:kfnumu}
(\widetilde{K_{1-\beta}f_{\nu\mu}})(u) =  \int_0^1 \frac{\widetilde{f_{\nu\mu}}(v)}{(1-u)^\beta \left(v+\frac{u}{1-u}\right)^\beta}\,\operatorname{d}\!v 
 = u^{-\beta} \cdot B(\mu,\nu) \cdot {}_2F_1\left(\beta,\nu;\mu+\nu;\frac{u-1}{u}\right),
\end{equation}
where $B$ denotes the beta function defined by $B(\mu,\nu)=\int_0^1 t^{\mu-1}(1-t)^{\nu-1}\,dt$ and satisfying $B(\mu,\nu)=\frac{\Gamma(\mu)\,\Gamma(\nu)}{\Gamma(\mu+\nu)}$ (see \cite[Formula 8.384(1)]{GR}), and ${}_2F_1$ the hypergeometric function defined by the following power series for $|z|<1$,  see \cite[9.101]{GR}:
\begin{multline*}
{}_2F_1(a,b;c;z) \ = \ 1 + \frac{a\cdot b}{c\cdot 1}\,z + \frac{a\,(a+1)\cdot b(b+1)}{c(c+1)\cdot 1\cdot 2}\,z^2 +  \frac{a\,(a+1)(a+2)\cdot b(b+1)(b+2)}{c(c+1)(c+2)\cdot 1\cdot 2\cdot 3}\,z^3 \\ +  \frac{a\,(a+1)(a+2)(a+3)\cdot b(b+1)(b+2)(b+3)}{c(c+1)(c+2)(c+3)\cdot 1\cdot 2\cdot 3\cdot 4}\,z^4 + \ldots. 
\end{multline*}
Since  ${}_2F_1(a,b;c;0)=1$, it follows immediately from Eq.~\eqref{eq:kfnumu} that 
\begin{equation*}
(\widetilde{K_{1-\beta}f_{\nu\mu}})(1) = B(\mu,\nu)=\frac{\Gamma(\mu)\,\Gamma(\nu)}{\Gamma(\mu+\nu)}.
\end{equation*}

Using the transformation formulae for the hypergeometric function in \cite[Formulae 9.132(1) \& (2)]{GR}, we obtain the following formulae which avoid having to work with the analytic continuation of ${}_2F_1$ explicitly for $u\in (0,1)$:
\begin{align*}
(\widetilde{K_{1-\beta}f_{\nu\mu}})(u) & =  
\frac{\Gamma(\mu)\Gamma(\nu-\beta)}{\Gamma(\mu+\nu-\beta)}\,  {}_2F_1\left(\beta,\mu;\beta-\nu+1;u\right) \notag \\ & \qquad + \ u^{\nu-\beta}\,\frac{\Gamma(\nu)\Gamma(\beta-\nu)}{\Gamma(\beta)}\, {}_2F_1\left(\nu,\mu+\nu-\beta;\nu-\beta+1;u\right)\label{eq:KFnumu}\\
& = \frac{\Gamma(\mu)\Gamma(\nu-\beta)\Gamma(\beta-\nu+1)\Gamma(1-\nu-\mu)}{\Gamma(\mu+\nu-\beta)\Gamma(1-\nu)\Gamma(1+\beta-\mu-\nu)}\, {}_2F_1\left(\beta,\mu;\mu+\nu;1-u\right)   \notag \\ & \qquad
 \quad + u^{\nu-\beta}\, \frac{\Gamma(\nu)\Gamma(\beta-\nu)\Gamma(\nu-\beta+1)\Gamma(1-\nu-\mu)}{\Gamma(\beta)\Gamma(1-\mu)\Gamma(1-\beta)}\, {}_2F_1\left(\nu,\mu+\nu-\beta;\mu+\nu;1-u\right).\notag
\end{align*}

From the convergence behaviour of ${}_2F_1$, see \cite[9.102]{GR}, we obtain from these calculations that $(\widetilde{K_{1-\beta}f_{\nu\mu}})(u)$ converges at $u=1$ (for $\mu,\nu>0$ and $0<\beta<1$); therefore, $(\widetilde{K_{1-\beta}f_{\nu\mu}})$ is in $C[0,1]$ if $\nu>\beta$, but diverges at $u=0$ and is continuous on $(0,1]$ if $\nu\le\beta$. In particular, we see that this integral operator ``smoothens'' the functions $f_{\nu\mu}$: the (possible) singularity at $u=1$ is no longer present in $(\widetilde{K_{1-\beta}f_{\nu\mu}})$, while the order of the (possible) singularity at $u=0$ is reduced from $u^{\nu-1}$ to $u^{\nu-\beta}$.

We now compare this result for $(\widetilde{K_{1-\beta}f_{\nu\mu}})(u)$ with the approximation used in the proof of Theorem~\ref{thm:main} in Section~\ref{sec:proof}. In Fig.~\ref{fig:case1n}, we set $\beta=\frac12=\alpha$ and consider $\widetilde{f_{1,\frac23}}(u)=(1-u)^{-1/3}$ which belongs to $L^p[0,1]$ for $p<3$. We have
\begin{equation*}
(\widetilde{K_{\frac12}f_{1,\frac23}})(u) = \frac32\,\sqrt{u}\cdot {}_2F_1\left(1,\frac76;\frac53;1-u\right)
\end{equation*}
and thus $(K_{\frac12}f_{1,\frac23})(x) = \frac32\,\sqrt{1-x}\cdot {}_2F_1(1,\frac76;\frac53;x)\in L^{\infty}[0,1]$ (note that $\frac32= \Gamma(\frac23)/\Gamma(\frac53) \le (K_{\frac12}f_{1,\frac23})(x) \le \sqrt{\pi}\cdot\Gamma(\frac23)/\Gamma(\frac76) \approx 2.587$ for $x\in[0,1]$). We then use Mathematica, see~\ref{sec:mathematica}, to obtain the approximation $\phi_n$ obtained in Section~\ref{sec:proof} for various values of $n$; e.g., for $n=2$ (used in the top left panel of Fig.~\ref{fig:case1n}) we get the approximation
\begin{equation*}
\phi_2(u) = \begin{cases}
1.870 - 1.341 u & \text{for}\ u\in [0,\frac14] \\
1.764-0.487 u +0.132 u^{-1/2} -0.004 u^{-3/2}  & \text{for}\ u\in (\frac14,\frac12] \\
1.458-0.225 u+0.278 u^{-1/2} -0.017 u^{-3/2} & \text{for}\ u\in (\frac12,1] \\
\end{cases}
\end{equation*}
(all numbers rounded to $3$ decimal places). In Fig.~\ref{fig:case1n} we show log-log-plots of the difference $(\widetilde{K_{\frac12}f_{1,\frac23}})(u)-\phi_n(u)$ for various values of $n$; note that in this logarithmic scale, the intervals $(2^{-(k+1)},2^{-k}]$ appear with the same ``length''. We observe how the approximation gets better as $n$ increases, on the one hand by decreasing the first interval $[0,2^{-n}]$ where the approximation is ``bad'', on the other hand by decreasing the difference overall on each interval.

\begin{figure}[tb]
\centerline{{\includegraphics[width=.4\textwidth,height=0.44\textwidth]{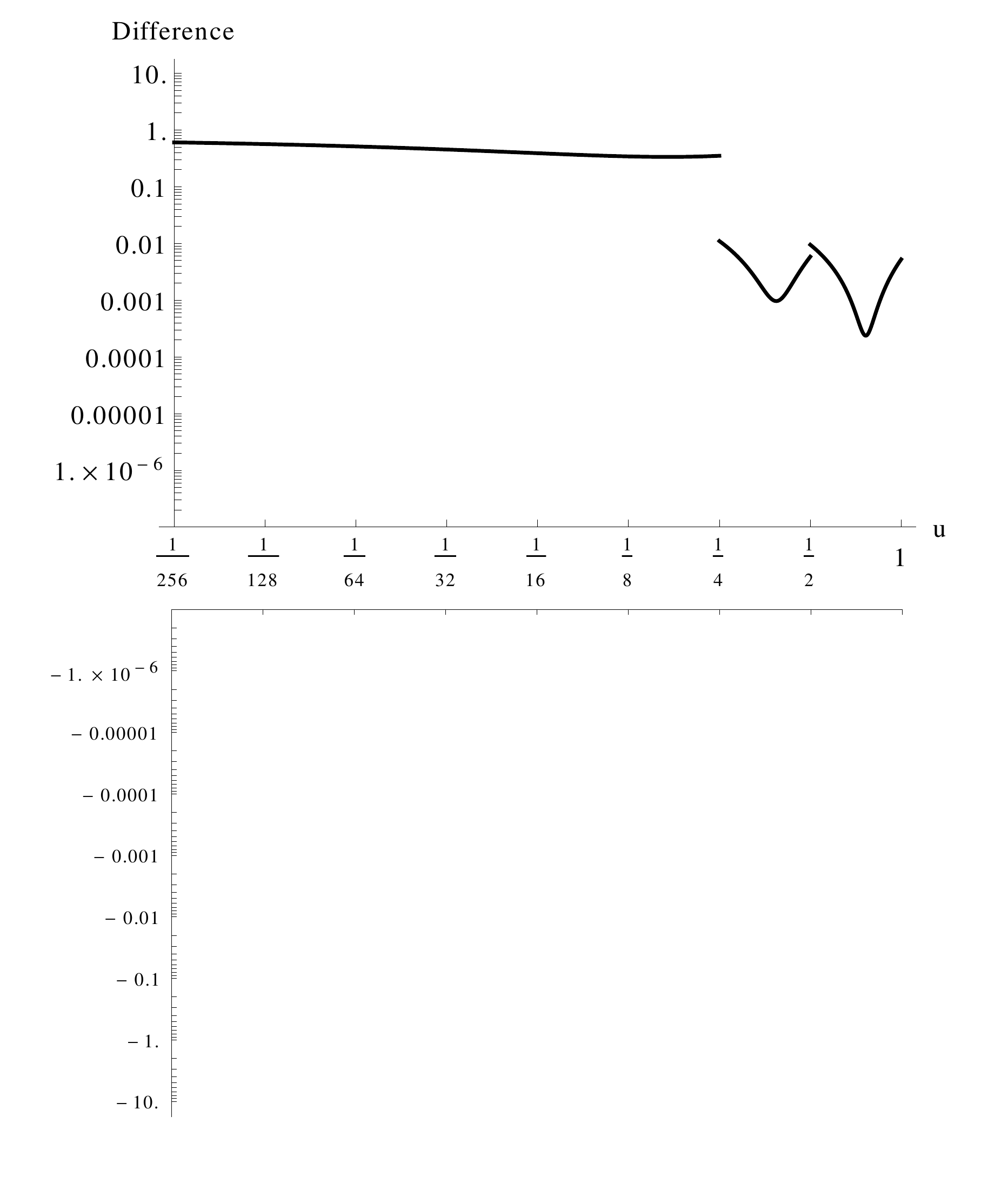}}\qquad{\includegraphics[width=.4\textwidth,height=0.44\textwidth]{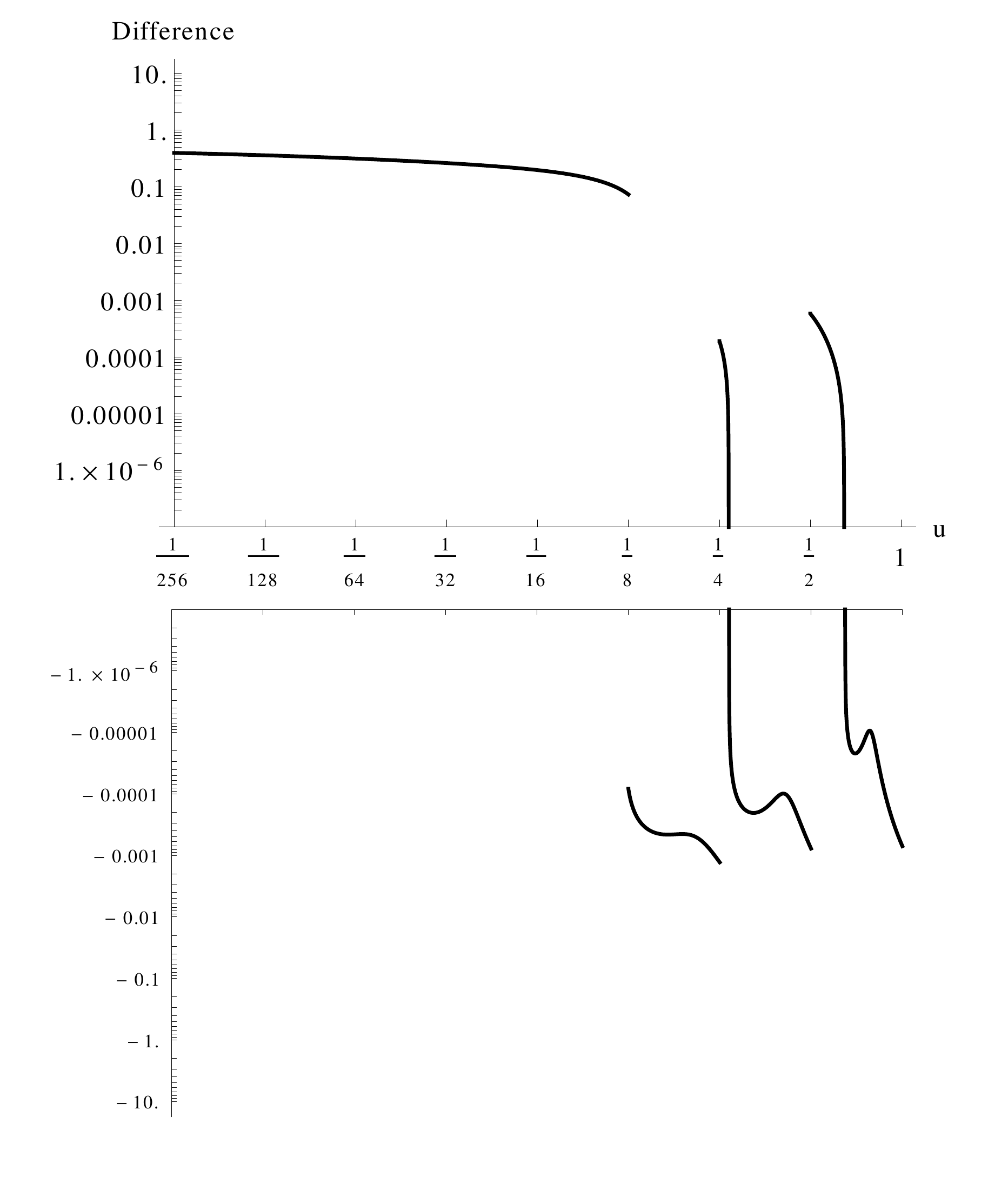}}}
\centerline{{\includegraphics[width=.4\textwidth,height=0.44\textwidth]{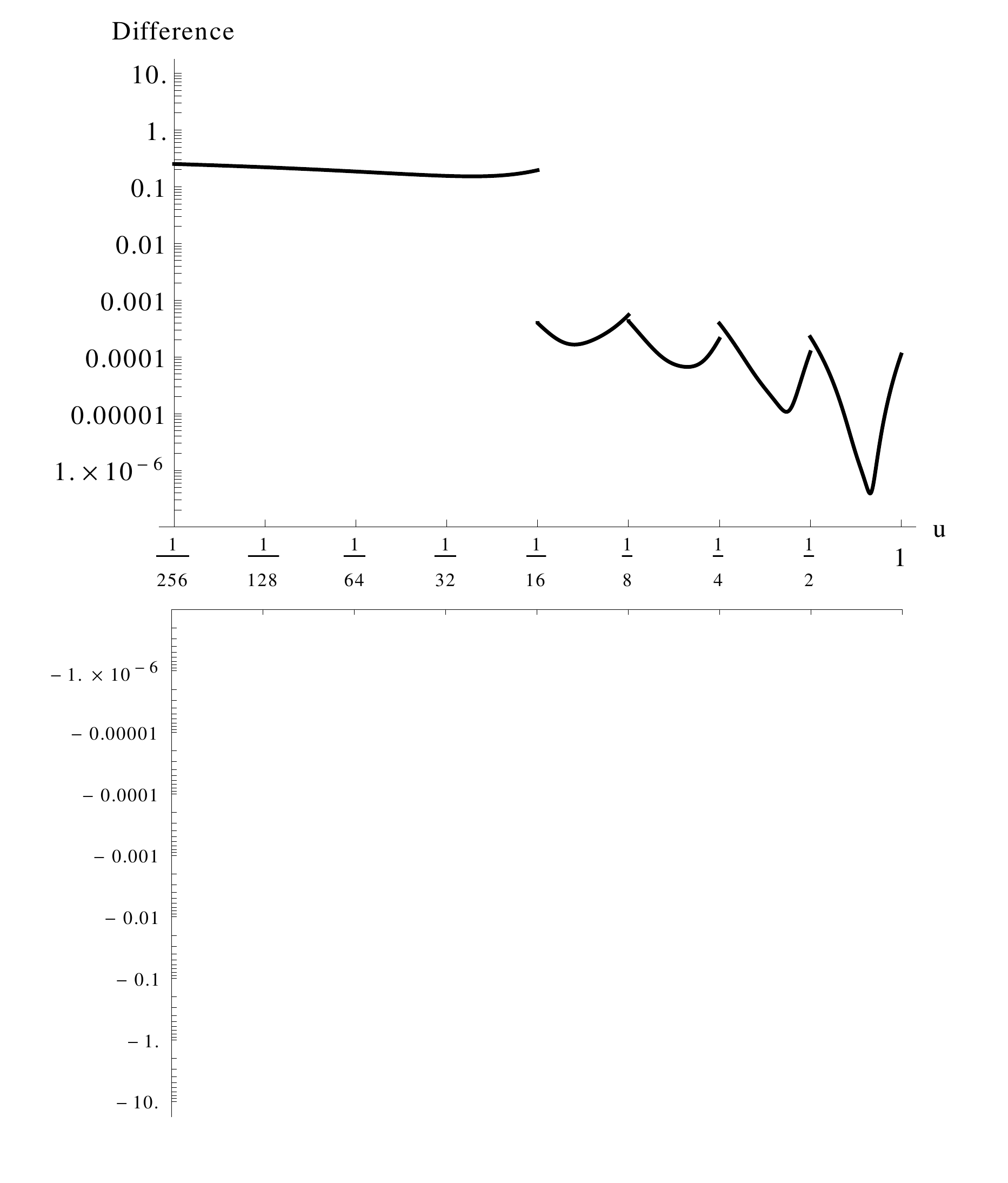}}\qquad{\includegraphics[width=.4\textwidth,height=0.44\textwidth]{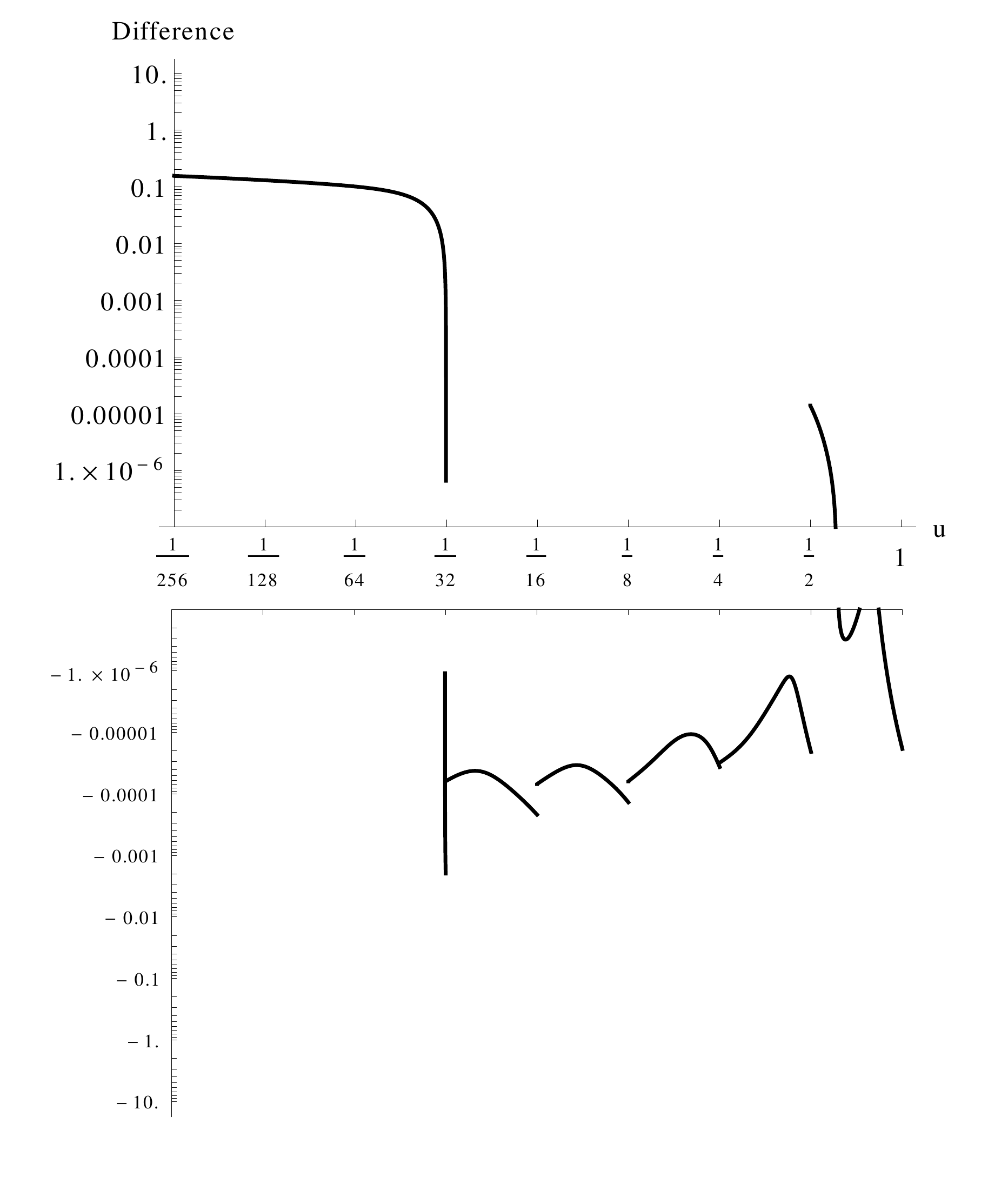}}}
\centerline{{\includegraphics[width=.4\textwidth,height=0.44\textwidth]{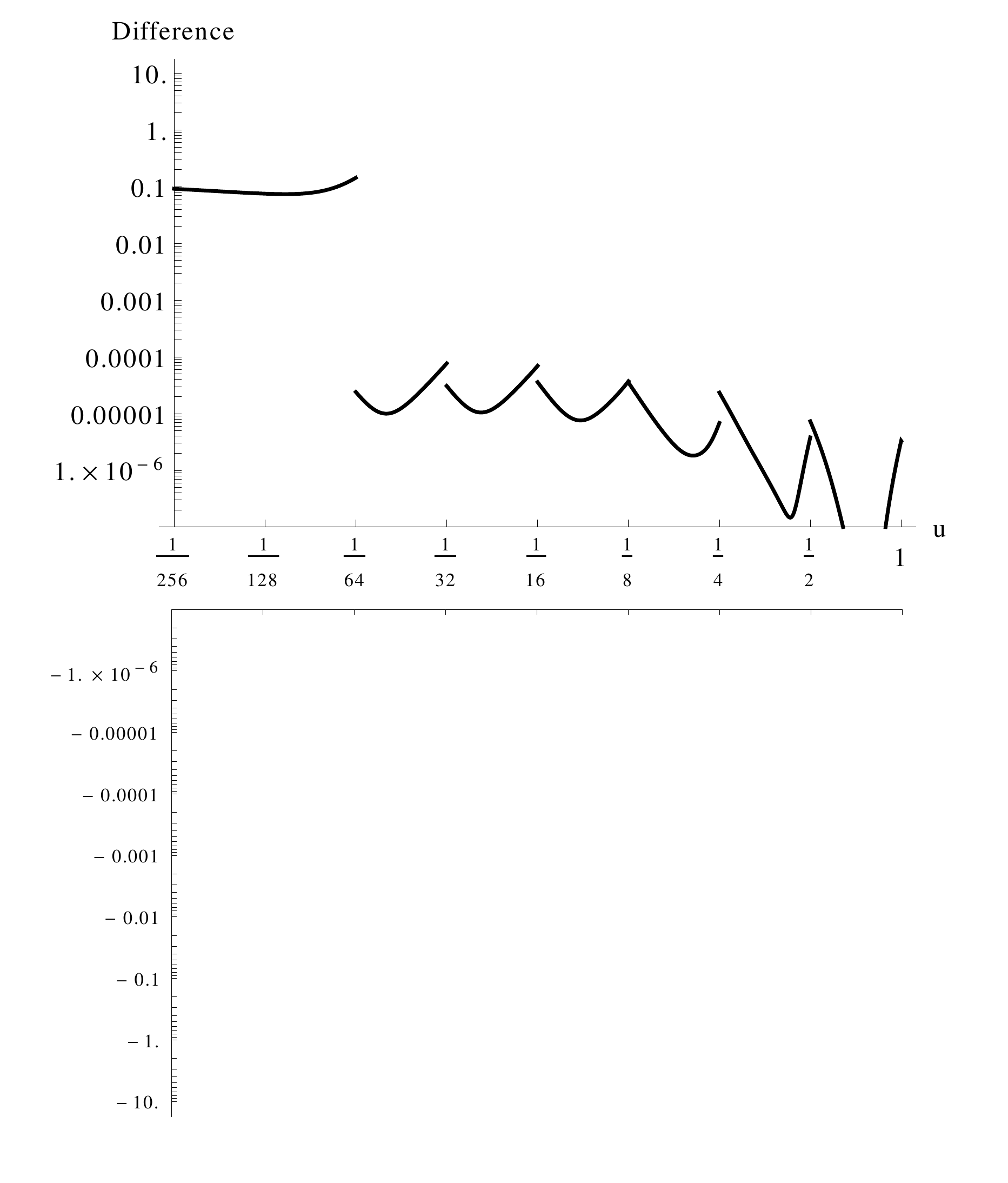}}\qquad{\includegraphics[width=.4\textwidth,height=0.44\textwidth]{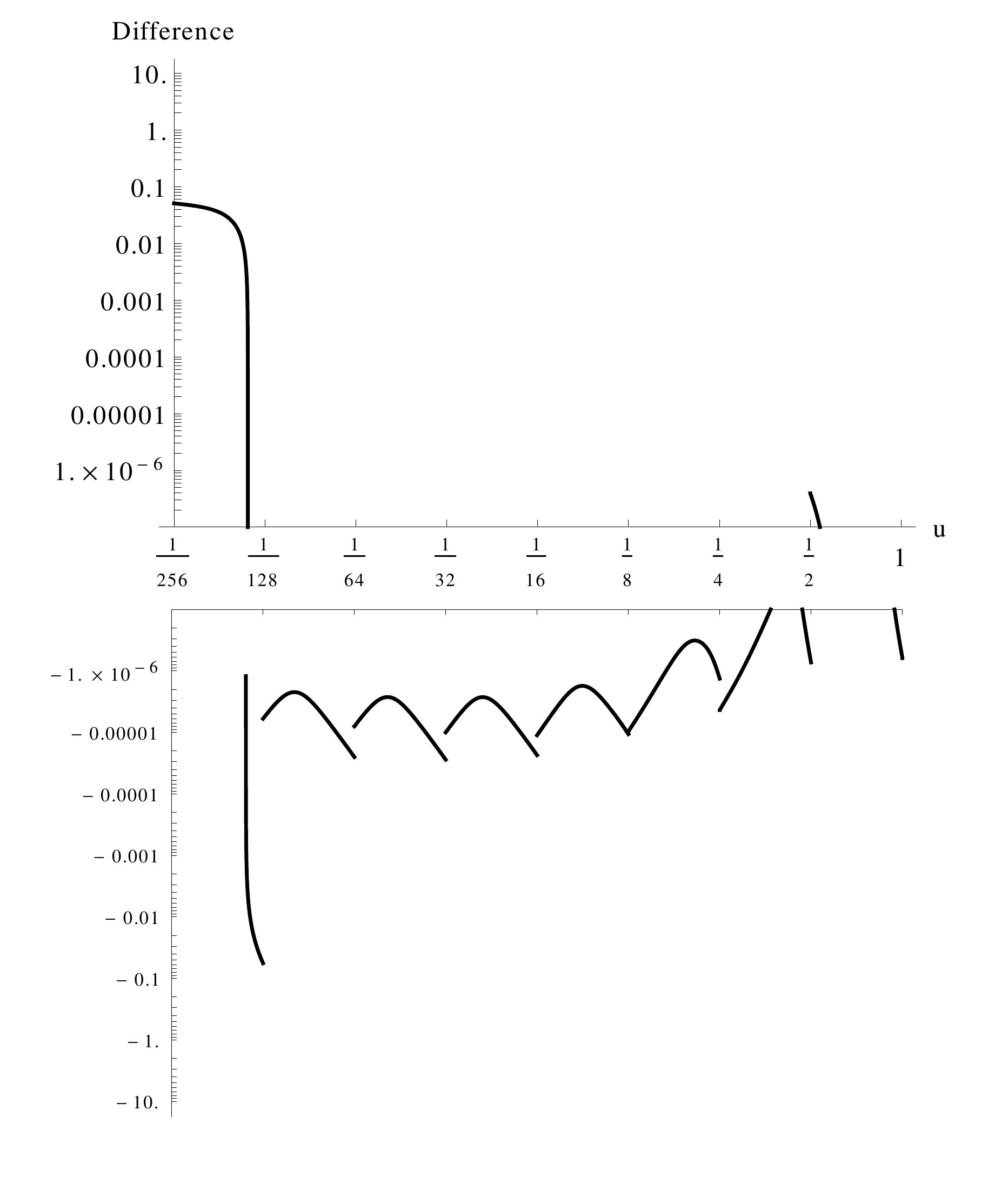}}}
\caption{Log-log plot of the difference $\left(\widetilde{K_{\frac12}f_{1,\frac23}}\right)(u)-\phi_n(u)$ where $\phi_n$ is the approximation we derived in this article, for $\beta=\frac12$ and $n=2$ (\textit{top, left}),  $n=3$ (\textit{top, right}),  $n=4$ (\textit{middle, left}),  $n=5$ (\textit{middle, right}),\  $n=6$ (\textit{bottom, left}), and $n=7$ (\textit{bottom, right}). \label{fig:case1n}}
\end{figure}

\section{Outlook and Conclusion}

The approximation we used to obtain upper bounds for the Kolmogorov widths with respect to $L^{\infty}[0,1]$ respectively $L^r[0,1]$ in Section~\ref{sec:proof} contain -- besides polynomial terms -- powers of the form $x^{\alpha-k}$ with $0<\alpha<1$ and $k\in\mathbb{N}$. While such powers don't have a good approximation by polynomials, they can be well approximated by rational functions, see \cite{N,B,G,V,S}; also compare \cite[Chapters 7 \& 8]{LGM}. In fact, the results in \cite{S} show that if we consider approximation by rational functions with both numerator and denominator of degree at most $n$, then the upper bound on the error of approximation of the rational best approximant of $x^\alpha$ with $\alpha>0$ is bounded by $O(e^{-c \sqrt{n}})$ for some constant $c>0$. Thus, combining this with our results on Kolmogorov widths indicates that we can approximate $K_{\alpha}$ by rational functions with the same asymptotic upper bounds on the  error of approximation. We will not do this here.

We obtained upper bounds on the Kolmogorov widths of $K_{\alpha}$; of course, the harder task is obtaining lower bounds. For the Hilbert-Schmid operator $K_{\alpha}: L^2[0,1]\to L^2[0,1]$, for which we calculated the operator norm exactly in Proposition~\ref{prop:HS1}, the Kolmogorov widths coincide with the singular values; furthermore, in this case $K_{\alpha}$ is a Schatten class operator, see Proposition~\ref{prop:KSchatten}. Can one use these facts to obtain lower bounds on the Kolmogorov widths at least in this case?

\appendix

\section{Mathematica Code}\label{sec:mathematica}

Mathematica will (try to) evaluate  $u\mapsto \left(\widetilde{K_{1-\beta}f}\right)(u)$ exactly using (here, with $\beta=\frac12$ and the constant function $f\equiv 1$):
\begin{verbatim}
> beta =1/2;
> f[x_]:=1;
> kf[u_] := Evaluate[Integrate[f[v] (u + v - u v)^(-beta), {v, 0, 1}]]
> Plot[kf[u], {u, 0, 1}]
\end{verbatim}

To calculate its approximation (denoted as function \texttt{app} in the following code), we use
\begin{itemize}
\item for $u\in [0,2^{-n}]$ (e.g., if we choose $n=5$):
\begin{verbatim} 
> n = 5; 
> p[x_] := Evaluate[Normal[Series[(1 + x)^(-beta), {x, 0, n - 1}]]]
> ser[u_, v_] := f[v] v^(-beta) p[(u - u*v)/v]
> app[u_] := Evaluate[Total[Table[NIntegrate[SeriesCoefficient[
         ser[u, v], {u, 0, k}], {v, 2^(-1 - n), 1}] u^k, 
         {k, 0, n - 1}]]]
\end{verbatim}
\item for each $k=0,\ldots, n-1$ and $u\in(2^{-(k+1)},2^{-k}]$ (e.g., for $k=2$):
\begin{verbatim} 
> k = 2
> ser1[z_, v_] := f[v] p[(v - v/z) z]
> app1[u_] := Evaluate[Total[Table[NIntegrate[SeriesCoefficient[
         ser1[z, v], {z, 0, j}], {v, 0, 2^(-k - 2)}] u^(-beta - j), 
         {j, 0, n - 1}]]]
> ff[u_] := Integrate[f[v]/(u + v - u v)^(beta),{v, 2^(-k - 2), 
         Min[2^(-k + 1), 1]}]
> du[u_] := Evaluate[Table[D[ff[u], {u, j}], {j, 0, n - 1}]];
> app2[u_] := Evaluate[Expand[du[2^(-k - 1) + 2^(-k - 2)].
         Table[(u - (2^(-k - 1) + 2^(-k - 2)))^j/j!, 
         {j, 0, n - 1}]]];
> ser3[u_, v_] := f[v] v^(-beta) p[(u - u*v)/v]
> app3[u_] := If[k > 1, Evaluate[Total[Table[NIntegrate[
          SeriesCoefficient[ser3[u, v], {u, 0, j}], 
          {v, 2^(1 - k), 1}] u^j, {j, 0, n - 1}]]], 0]
> app[u_]:=Evaluate[app1[u] + app2[u] + app3[u]]
\end{verbatim}
\end{itemize}

This code was used to calculate the difference of $(\widetilde{K_{\frac12}f_{1,\frac23}})(u)$ to its approximation in Fig.~\ref{fig:case1n}.

\section*{Acknowledgements}

Eduard Belinsky posed the question of estimating the Kolmogorov widths of $K_{\alpha}$ to the first author for his master thesis. Eduard Belinsky worked on estimates for entropy numbers \cite{Bel98,Bel02,BT} as well as Kolmogorov widths \cite{BL2,BL} and their relationship to each other \cite{TB}, also see \cite{Lacey}. Unfortunately, Eduard Belinsky passed away before this work was completed, which then only happened years later. This paper is our way to say thank you to Eduard Belinsky.


\begin{thebibliography}{00}

\bibitem{Bel98} Belinsky, E. 1998. ``Estimates of entropy numbers and Gaussian measures for classes of functions with bounded mixed derivative''. \emph{J.\ Approx.\ Th.}\ \textbf{93}: 114--127.

\bibitem{Bel02} Belinsky, E. 2002. ``Entropy numbers of vector-valued diagonal operators''. \emph{J.\ Approx.\ Th.}\ \textbf{117}: 132--139.

\bibitem{BL2} Belinsky, E.,\ and W.\ Linde. 2002. ``Small ball probabilities of fractional Brownian sheets via fractional integration operators''. \emph{J.\ Theoret.\ Prob.}\ \textbf{15}(3): 589--612.

\bibitem{BL} Belinsky, E.,\ and W.\ Linde. 2006. ``Compactness properties of certain integral operators related to fractional integration''. \emph{Math.\ Z.}\ \textbf{252}(3): 669--686.

\bibitem{BT} Belinsky, E.,\ and W.~Trebels. 2005. ``Almost optimal estimates for entropy numbers of $B_{2,2}$ and its
consequences''. \emph{Math.\ Z.} \textbf{250}: 23--42.

\bibitem{BS2} Birman, M.Sh.,\ and M.Z.\ Solomyak. 1971. ``Asymptotic behavior of the spectrum of weakly polar integral operators''. 
\emph{Math.\ USSR, Izvestiya} \textbf{4}(1970): 1151--1168. 
(Translation from \emph{Izv.\ Akad.\ Nauk SSSR Ser.\ Mat.}\ \textbf{34}(5): 1142--1158.)

\bibitem{BS} Birman, M.Sh.,\ and M.Z.\ Solomyak. 1977. ``Estimates of singular numbers of integral operators''. \emph{Russ.\ Math.\ Surv.}\ \textbf{32}(1): 15--89.
(Translation from \emph{Uspehi mat.\ Nauk} \textbf{32}(1): 17--84.)

\bibitem{B} Bulanov, A. P. 1971. ``On the order of approximation of convex functions by rational functions''. \emph{Math.\ USSR, Izvestiya}\ 3(1969): 1067-1080.
(Translation from \emph{Izv.\ Akad.\ Nauk SSSR, Ser.\ Mat.}\ \textbf{33}: 1132-1148 (1969).)

\bibitem{Carl81} Carl, B. 1981. ``Entropy numbers, s-numbers, and eigenvalue problems''. \emph{J.\ Funct.\ Anal.}\ \textbf{41}: 290--306.

\bibitem{CS90} Carl, B., and I.~Stephani. 1990. \emph{Entropy, Compactness and the Approximation of Operators}. Cambridge: Cambridge.

\bibitem{DL} DeVore, R.A.,\ and G.G.~Lorentz. 1991. \emph{Constructive Approximation}. Heidelberg: Springer.

\bibitem{ET} Edmunds, D.E.,\ and H.~Triebel. 1996. \emph{Function Spaces, Entropy Numbers and Differential Operators}. Cambridge: Cambridge.

\bibitem{G} Ganelius,  T. 1979. ``Rational approximation to $x^{\alpha}$ on $[0,1]$''. \emph{Anal.~Math.} \textbf{5}: 19--33.

\bibitem{GR} Gradshteyn,  I.S., and I.M.~Ryzhik, 2000. \emph{Table of Integrals, Series, and Products}, 6th edition. Edited by A.~Jeffrey and D.~Zwillinger. San Diego, CA: Academic Press. 

\bibitem{K} Kolmogorov, A.N. 1936.``\"{U}ber die beste Ann\"{a}herung von Funktionen einer gegebenen  Funktionenklasse''. \emph{Ann.\ of Math}\ \textbf{37}(2): 216--221.

\bibitem{Kol56} Kolmogorov, A.N. 1956. ``On certain asymptotic characteristics of completely bounded metric spaces''. \emph{Dokl.\ Akad.\ Nauk SSSR}\ \textbf{108}: 385--388.

\bibitem{Kol58} Kolmogorov, A.N. 1958. ``On linear dimensionality of topological vector spaces''. \emph{Dokl.\ Akad.\ Nauk SSSR}\ \textbf{120}: 239--241.

\bibitem{KT59} Kolmogorov, A.N., and V.M.~Tikhomirov. 1961. ``$\varepsilon$-entropy and $\varepsilon$-capacity of sets in function spaces''. \emph{Amer.\ Math.\ Soc., Translat.\ II. Ser.}\ \textbf{17}: 227--364.
(Translation from \emph{Uspehi mat.\ Nauk} \textbf{14}(2): 3--86 (1959).)

\bibitem{Kor} Korne\u{\i}chuk, N.P. 1991. \emph{Exact Constants in Approximation Theory}. Cambridge: Cambridge.

\bibitem{Lacey} Lacey, M. 2008. ``Kolmogorov entropy of the mixed derivative spaces''. Available at \texttt{http://www.math.technion.ac.il/hat/articles.html}.

\bibitem{Lap} Laptev, A. 1974. ``Spectral asymptotic behavior of a class of integral operators''. \emph{Math.\ Notes} \textbf{16}: 1038--1043.
(Translation from \emph{Mat.\ Zametki}\ \textbf{16}: 741-750.)

\bibitem{Lebedev} Lebedev, N.N. 1972. \emph{Special Functions \& Their Applications}. New York, NY: Dover.

\bibitem{Lor} Lorentz,  G.G. 1966. ``Metric entropy and approximation''. \emph{Bull.\ Amer.\ Math.\ Soc.}\ \textbf{72}: 903--937. 

\bibitem{L} Lorentz,  G.G. 1986. \emph{Approximation of Functions}. New York, NY: Chelsea.

\bibitem{LGM} Lorentz, G.G.,  M.v.~Golitschek and Y.~Makovoz. 1996. \emph{Constructive Approximation, Advanced Problems}. Heidelberg: Springer.

\bibitem{Lue} Luecking, D.H. 1987. ``Trace ideal criteria for Toeplitz operators''. \emph{J.\ Funct.\ Anal.},\ \textbf{73}(2): 345--368. 

\bibitem{N} Newman, D.J. 1964. ``Rational approximation of $|x|$''. \emph{Michigan Math.\ J.}\ \textbf{11}: 11--14.

\bibitem{NIST} \textit{NIST Digital Library of Mathematical Functions}, \texttt{http://dlmf.nist.gov/}, Release 1.0.11 of 2016-06-08. Online companion to \cite{NIST2}.

\bibitem{NIST2} F.W.J. Olver, D.W.~Lozier, R.F.~Boisvert, and C.W.~Clark (eds.), \textit{NIST Handbook of Mathematical Functions}, Cambridge, New York, NY, 2010. Print companion to \cite{NIST}.

\bibitem{P} Parfenov, O.G. 1992. ``Estimates of singular values of integral operators with analytic kernels''. \emph{Vestn.\ St.~Petersburg Univ., Math.}\ \textbf{25}(2): 23--31. 
(Translation from \emph{Vestn.\ St.~Petersburg Univ.., Ser.~I, Mat.\ Mekh.\ Astron.}\ \textbf{1992}(2): 24--32.) 

\bibitem{Pietsch} Pietsch, A. 1978. \emph{Operator Ideals}. Berlin: VEB.

\bibitem{Pietsch2} Pietsch, A. 1987. \emph{Eigenvalues and $s$-Numbers}. Cambridge: Cambridge.

\bibitem{Pisier} Pisier, G. 1989. \emph{The Volume of Convex Bodies and Banach Space Geometry}. Cambridge: Cambridge.

\bibitem{Pinkus} Pinkus, A. 1985. \emph{$n$-Widths in Approximation Theory}. Berlin: Springer.

\bibitem{Pin86} Pinkus, A. 1986. ``$n$-widths and optimal recovery''. In: de Boor, C.\ (ed.). \emph{Approximation Theory (Proceedings of Symposia in Applied Mathematics, Volume 36)}. Providence, RI: AMS. pp.~51--66.  

\bibitem{Pin03} Pinkus, A. 2003. ``Negative theorems in approximation theory''. \emph{American Math.\ Monthly} \textbf{110}(10): 900--911.

\bibitem{Rai60} Rainville, E.A.1960. \emph{Special Functions}. New York, NY: Macmillan.

\bibitem{Simon} Simon, B. 2015. \emph{Operator Theory}. Providence, RI: AMS.

\bibitem{S} Stahl, H.R. 2003. ``Best Uniform Rational Approximation of $x^\alpha$ on $[0,1]$''. \emph{Acta Math.}\ \textbf{190}: 241--306.

\bibitem{Tik90} Tikhomirov, V.M. 1990. ``Approximation theory''. In: R.V.~Gamkrelidze (ed.). \emph{Analysis II: Convex Analysis and Approximation Theory}. Berlin: Springer. pp.~93--243.

\bibitem{TB} Trigub, R.M., and E.S.~Belinsky. 2004. \emph{Fourier Analysis and Approximation of Functions}. Dordrecht: Kluwer.

\bibitem{Vit61} Vitushkin, A.G. 1961. \emph{Theory of the transmission and processing of information}. New York, NY: Pergamon Press.

\bibitem{V} Vyacheslavov, N.S. 1981. ``On the approximation of $x^\alpha$ by rational functions''. \emph{Math.\ USSR, Izv.}\ \textbf{16}(1980): 83--101. 
(Translation from \emph{Izv.\ Akad.\ Nauk SSSR, Ser.\ Mat.}\ \textbf{44}(1):  92--109  (1980).)

\bibitem{W} Widom, H. 1964. ``Asymptotic behavior of the eigenvalues of certain integral equations II''. \emph{Arch.\ Rat.\ Mech.\ Anal.}\ \textbf{17}(3): 215--229.

\bibitem{Z}  Zhu, K. 1990. \emph{Operator Theory in Function Spaces}. New York, NY: Marcel Dekker.
 

\end{thebibliography}
\end{document}